\documentclass[12pt,reqno]{amsart}

\usepackage[margin=2cm]{geometry}

\usepackage{amsmath,amssymb,mathrsfs,framed,slashed,url}
\usepackage[margin=2cm]{geometry}
\usepackage{amsmath}
\usepackage{amscd}
\usepackage{amssymb}
\usepackage{dsfont}
\usepackage{braket}
\usepackage{empheq}
\usepackage{enumerate}
\usepackage{mathrsfs}
\usepackage{comment}
\usepackage{color}
\usepackage{soul}
\usepackage{framed}
\usepackage[color=yellow]{todonotes}
\usepackage{graphicx}
\usepackage{subfigure}
\usepackage[colorlinks]{hyperref}
\usepackage{url}

\newtheorem{theorem}{Theorem}[section]

\newtheorem{definition}{Definition}[section]

\newtheorem{lemma}[theorem]{Lemma}

\newtheorem{remark}[theorem]{Remark}

\newtheorem*{theorem*}{Theorem}

\numberwithin{equation}{section}

\DeclareMathOperator{\diff}{d\!}

\begin{document}

\title[Kunita-It\^o-Wentzell formula for $k$-forms in stochastic fluid dynamics]{ Implications of Kunita-It\^o-Wentzell formula for $k$-forms\\ in stochastic fluid dynamics }

\author[]{Aythami Bethencourt de L\'eon, Darryl Holm, Erwin Luesink, So Takao}
\address{Mathematics Department, Imperial College London}
\email{a.bethencourt-de-leon15@ic.ac.uk,d.holm@ic.ac.uk}
\email{e.luesink16@ic.ac.uk,st4312@ic.ac.uk}

\maketitle

\begin{abstract}
We extend the It\^o-Wentzell formula for the evolution of a time-dependent stochastic field along a semimartingale to $k$-form-valued stochastic processes. The result is the Kunita-It\^o-Wentzell (KIW) formula for $k$-forms. We also establish a correspondence between the KIW formula for $k$-forms derived here and a certain class of stochastic fluid dynamics models which preserve the geometric structure of deterministic ideal fluid dynamics. This geometric structure includes  Eulerian and Lagrangian variational principles, Lie--Poisson Hamiltonian formulations and natural analogues of the Kelvin circulation theorem, all derived in the stochastic setting.

\end{abstract}

\setcounter{tocdepth}{1}
\tableofcontents

\section{Introduction}

\paragraph{\bf Purpose of this paper.} This paper aims to derive stochastic partial differential equations (SPDEs) for continuum dynamics with stochastic advective Lie transport (SALT). These derivations require stochastic counterparts of the deterministic approaches for deriving fluid equations (PDEs). The approach we follow is the stochastic counterpart of the  Euler--Poincar\'e variational principle as in \cite{holm1998euler} which reveals the geometric structure of deterministic ideal fluid dynamics. Our goal is to create stochastic counterparts which preserve this geometric structure. Such variational formulations of stochastic fluid PDEs will also possess auxiliary SPDEs for stochastic advection of material properties which will correspond to various differential $k$-forms. The auxiliary SPDE for advective transport of a given material property by a stochastic fluid flow corresponds to a type of stochastic ``chain rule" in which a $k$-form-valued semimartingale $K(t,x)$ is evaluated (via pull-back) along a stochastic flow $\phi_t$. The stochastic aspect of the flow map $\phi_t$ represents uncertainty
in the Lagrange-to-Euler map for the fluid. Examples of the $k$-form valued process $K(t,x)$ advected by the pull-back  $\phi_t^*K$ of the stochastic flow map $\phi_t$  include mass density, regarded as a volume form, and magnetic field, interpreted as a two-form for ideal magnetohydrodynamics (MHD) \cite{holm1998euler}. 
%
We will refer to this ``chain rule" that gives us the auxiliary SPDE for SALT satisfied by the $k$-form-valued process $\phi_t^* K$ as  the \emph{Kunita--It\^o--Wentzell (KIW) formula} for $k$-forms.

In \cite{kunita1981some,kunita1997stochastic}, Kunita derived an extension of It\^o's formula showing that if 
$K$ is a sufficiently regular $(l,m)$-tensor and $\phi_{t}$ is the flow of the SDE 
\begin{align} \label{phi-map-intro}
\diff \phi_{t}(x) = b(t,\phi_{t}(x)) \diff t + \xi(t,\phi_{t}(x)) \circ \diff B_t
\,,
\end{align}
with sufficently regular coefficients, then an analogue of It\^o's formula holds for tensor fields, namely
\begin{align} 
&\phi_{t}^* K(t,x)-K(0,x) = \int_0^t \phi_{s}^* (\mathcal{L}_{b} K)(s,x) \diff s
+  \int_0^t \phi_{s}^* (\mathcal{L}_{\xi} K)(s,x) \circ \diff B_s\,.
\label{ito-second-1}
\end{align}
Here, $\circ \diff B_s$ denotes Stratonovich integration with respect to the Brownian motion $B_s,$ which is defined with respect to the standard probability space and filtration. In addition, $\phi_t^*(\,\cdot\,)$ denotes the pullback with respect to the map $\phi_t$ and $\mathcal L_b$ denotes the Lie derivative with respect to the vector field $b$. In \cite{krylov2011ito}, Krylov considered an approach using mollifiers to provide a general proof of the classical It\^o-Wentzell formula  \cite{kunita1981some,Bismut1981-ItoFormula,kunita1997stochastic}. This classical formula states that for a sufficiently smooth scalar function-valued semimartingale $f$, represented as
\begin{align} \label{Ito-Wentzell-intro1}
\diff \!f(t,x) =  g(t,x)\diff t +  h(t,x) \circ \diff W_t
\,,\end{align}
the equation describing the evolution of the function $f$ via pull-back by $\phi_t$ as $\phi_t^*f := f \circ \phi_t$ reads
\begin{align} \label{Ito-Wentzell-intro}
\begin{split}
    \phi_t^*f(t,x) &=  f(0,x) + \int^t_0 \phi_s^* g(s,x) \,\diff s +   \int^t_0 \phi_s^* h(s,x) \circ \diff W_s \\
    &+ \int^t_0 \phi_s^* \big( b \cdot \nabla f \big)(s,x) \,\diff s 
    + \int^t_0 \phi_s^* \big( \xi \cdot \nabla f \big)(s,x) \circ \diff B_s\,,
\end{split}
\end{align}
where $\phi_t$ is the flow of the SDE in \eqref{phi-map-intro}. For more a precise statement of the regularity conditions, see \cite{krylov2011ito}. In the present paper, we derive the Kunita-It\^o-Wentzell Theorem which establishes the formula for the evolution of a $k$-form-valued process $\phi_t^*K$. This result generalises Kunita's formula \eqref{ito-second-1} and the It\^o--Wentzell formula for a scalar function \eqref{Ito-Wentzell-intro} by allowing $K$ to be any smooth-in-space, stochastic-in-time $k$-form on $\mathbb{R}^n$. Omitting the technical regularity assumptions provided in the more detailed statement of the theorem in Section \ref{sec-KIW-for-k-forms}, we now state a simplified version of our main theorem, as follows.

\begin{theorem*}[Kunita-It\^o-Wentzell formula for $k$-forms, simplified version]
Consider a sufficiently smooth $k$-form $K(t,x)$ in space which is a semimartingale in time
\begin{align} \label{spde-compact}
\diff K(t,x) = G(t,x) \diff t + \sum_{i=1}^M H_i(t,x) \circ \diff W_t^i,
\end{align}
where $W_t^i$ are i.i.d. Brownian motions. Let $\phi_t$ be a sufficiently smooth flow satisfying the SDE 
\[
\diff \phi_{t}(x) = b(t,\phi_{t}(x)) \diff t + \sum_{i=1}^N \xi_i(t,\phi_{t}(x)) \circ \diff B_t^i
\,,\]
in which  $B_t^i$ are i.i.d. Brownian motions. Then the pull-back $\phi_t^* K$ satisfies the formula
\begin{align}
\begin{split}
\diff \,(\phi_t^* K)(t,x) = &\phi_t^* G(t,x) \diff t + \sum_{i=1}^M \phi_t^* H_i(t,x) \circ \diff W_t^i \\
&+  \phi_t^*\mathcal L_b K(t,x) \diff t  + \sum_{i=1}^N \phi_t^* \mathcal L_{\xi_i} K(t,x) \circ \diff B_t^i.
\label{KIWkformsimplified}
\end{split}
\end{align}
\end{theorem*}
Formulas \eqref{spde-compact} and \eqref{KIWkformsimplified} are compact forms of equations \eqref{K-eq-strat} and \eqref{Ito-Wentzell-one-form-Strat-ver} in Section \ref{sec-KIW-for-k-forms}. The latter equations are written in integral notation to make the stochastic processes more explicit.
\begin{remark}\rm
In applications, we will sometimes express \eqref{KIWkformsimplified} using the differential notation
\begin{align*}
\diff \,(\phi_t^* K)(t,x) = \phi_t^*\left(\diff K + \mathcal L_{\diff x_t} K\right)(t,x),
\end{align*}
where $\diff x_t$ is the stochastic vector field $\diff x_t(x) = b(t,x) \diff t + \sum_{i=1}^N \xi_i(t,x) \circ \diff B_t^i$.
This formula is also valid when $K$ is a vector field rather than a $k$-form.
\end{remark}

A quick comparison of the It\^o--Wentzell formulas in \eqref{Ito-Wentzell-intro} and \eqref{KIWkformsimplified} shows the parallels and differences between the scalar and $k$-form cases. Our proof of this theorem relies on a slight extension of Krylov's mollifier approach in \cite{krylov2011ito}. Our proof uses mollifiers to evaluate the time-dependent $k$-form $K(t,x)$ along the flow $\phi_t$ without having to discretise the time and take limits, as is usually done.
The result for deterministic, smooth-in-time $K$ is already available in \cite{kunita1984stochastic},
and for the particular case in which $K$ is a deterministic $k$-form-valued process, some consequences in fluid dynamics have also been discussed previously in \cite{catuogno2016stochastic,rezakhanlou2016stochastically}. In a related work, Drivas and Holm \cite{drivas2018circulation} prove the KIW theorem for one-forms in the course of proving Kelvin's circulation theorem rigorously for stochastic fluids. The approach of \cite{drivas2018circulation} converts the line integral of a one-form along a closed circulation loop to a Riemann integral by parametrising the loop and then it applies the standard It\^o-Wentzell formula. The method employed in the present work does not depend on parametrising the surface over which the integral is taken. Consequently, our mollifier approach based on \cite{krylov2011ito} allows for a natural coordinate-free generalisation of the It\^o-Wentzell formula to $k$-forms.

\subsection*{Background. }
Stochastic geometric mechanics for continuum dynamics has recently had a sequence of developments, which we now briefly sketch.  \smallskip

\paragraph{\bf  Stochastic geometric mechanics.}
In \cite{holm2015variational}, the extension of geometric mechanics to include stochasticity in nonlinear fluid theories was accomplished by applying Hamilton's variational principle, constrained by using the Clebsch approach to enforce stochastic Lagrangian fluid trajectories arising from the stochastic Eulerian vector field
\begin{equation}
\diff x_t(x,t) := u(x,t)\, \diff t + \sum_{i=1}^N \xi_i (x) \circ \diff W^i(t) \,,
\label{VF-decomp}
\end{equation}
regarded as a decomposition into a drift velocity $u(x,t)$ and a sum over independent stochastic terms. 
Imposing this decomposition as a constraint on the variations in Hamilton's  principle for fluid dynamics \cite{holm1998euler}, led in \cite{holm2015variational} to new stochastic partial differential equation (SPDE) models which serve to represent the effects of unknown, rapidly fluctuating scales of motion on slower resolvable time scales in a variety of fluid theories, and particularly in geophysical fluid dynamics (GFD).  \smallskip

\paragraph{\bf  Analytical properties of stochastic fluid equations.}

One should expect that the properties of the fluid equations with stochastic transport noise as formulated in \cite{holm2015variational} should closely track the properties of the unapproximated solutions of the fluid equations. For example, if the unapproximated model equations are Hamiltonian, then the model equations with stochastic transport noise should also be Hamiltonian, as shown in \cite{holm2015variational}. In addition, local well-posedness in regular Sobolev spaces and a Beale-Kato-Majda blow-up criterion were proved in \cite{crisan2017solution} for the stochastic model of the 3D Euler fluid equation for incompressible flow derived in \cite{holm2015variational}.  \smallskip


\paragraph{\bf  Fluid flow velocity decomposition.}
The same decomposition of the fluid flow velocity into a sum of drift and stochastic parts derived in \cite{holm2015variational} was also discovered in \cite{CoGoHo2017} to arise in a multi-scale decomposition of the deterministic Lagrange-to-Euler flow map into a slow large-scale mean and a rapidly fluctuating small scale map. Homogenisation theory was used to derive effective slow stochastic particle dynamics for the resolved mean part, thereby justifying the stochastic fluid partial differential equations in the Eulerian formulation. 
The results of \cite{CoGoHo2017} justified regarding the Eulerian vector field in \eqref{VF-decomp} as a genuine decomposition of the fluid velocity into a sum of drift and stochastic parts, rather than simply as a perturbation of the dynamics meant to model unknown effects in uncertainty quantification. This result implied that the velocity decomposition \eqref{VF-decomp} could be used in parallel with data assimilation for the purpose of reduction of uncertainty. \smallskip

\smallskip

\paragraph{\bf  The main content of this paper.}$\,$
\begin{enumerate}[$\lozenge$]
\item
Section \ref{sec-StochEPthm} uses the Kunita-It\^o-Wentzell (KIW) formula for $k$-forms as a crucial element in proving the Euler--Poincar\'e theorem and the Clebsch Hamilton's principle for deriving the equations of stochastic continuum dynamics. 
These two stochastic variational approaches each recover the stochastic transport versions of all of the deterministic continuum dynamics models with advected quantities derived in \cite{holm1998euler}. They also confirm the stochastic continuum dynamics equations  derived in \cite{holm2015variational}. The case of stochastic compressible adiabatic magnetohydrodynamics (MHD) is presented as a new illustrative example of the power of this method for continuum dynamics with a variety with forces depending on several advected $k$-forms. 

\item
Section \ref{sec-KIW-for-k-forms}  summarises our main theorem, which derives the KIW formula, thereby extending the It\^o--Wentzell formula to stochastic $k$-form-valued processes. A brief sketch of the proof is also outlined. 

\item
Section \ref{sec-consequences} explains some implications of the KIW formula for stochastic fluid dynamics. These implications include stochastic advection by Lie transport of $k$-forms, including details of the derivations of the  continuity equation and Kelvin circulation theorem for stochastic fluid flows.

\item
Section \ref{sec-proofs} carries out the detailed  proof of the KIW formula in the It\^o representation.

\item
Section \ref{sec-conclusions} concludes the paper with a brief summary and some outlook for further research. 

\end{enumerate}


\section{Stochastic continuum Euler--Poincar\'e theorem} \label{sec-StochEPthm}
Following \cite{arnold1966principe}, we consider the Lagrangian trajectories of ideal continuum flows as time-dependent curves $x_t$ on a manifold without boundary $M.$ These curves are generated by the action $x_t=\phi_t X$ of a curve on the manifold of diffeomorphisms $\phi_t$ parameterised by time $t$ such that $X=\phi_0 X$ at time $t=0$. Inspired by related results of \cite{arnaudon2014stochastic,holm2015variational,chen2015constrained},
we examine a family of stochastic paths (Lagrangian trajectories) generated by the action $x_t=\phi_t X$ of the diffeomorphism $\phi_t $  on the manifold $M$ where $\phi_t $ is stochastic and given by the Stratonovich stochastic process
\begin{align}\label{stochdiff-Lag-traj}
\diff \phi_t(X) = u(t,\phi_t(X)) {\,\diff t} + \xi(t,\phi_t(X)) \circ {\diff W_t}
\,,
\end{align}
in which notation for the probability variable $\omega$ has been suppressed, and the subscript $t$ in $\phi_t$, for example, denotes explicit time-dependence, not partial time derivative. Equation \eqref{stochdiff-Lag-traj} is written in differential notation for a Stratonovich stochastic process, as explained, e.g., in \cite{lazaro2007stochastic,cruzeiro2018momentum}. Namely, equation \eqref{stochdiff-Lag-traj} is short notation for the sum of Stratonovich stochastic integrals:
\begin{align}\label{stochdiff-Lag-int1}
x_t - X = \phi_t (X) - \phi_0(X) =  \int_0^t \circ \diff \phi_s(X) \,  = \int_0^t u_s(x_s) {\,\diff s} 
+  \int_0^t \xi(x_s) \circ {\diff W_s}
\,.
\end{align}

\subsection{Stochastic continuum Euler--Poincar\'e theorem with advected quantities}$\,$\\
In preparation for introducing a stochastic version of the Euler--Poincar\'e variational principle for deterministic continuum dynamics established in \cite{holm1998euler}, we consider next a family of smooth pathwise deformations of the action $x_t=\phi_t X$ by a \emph{second family of diffeomorphisms}. The second family of diffeomorphisms is deterministic and is parameterised by $\varepsilon$, with $\varepsilon=0$ at the identity. We take the combined action of the two diffeomorphisms to be a single two-parameter family, whose action on the flow manifold $M$ is denoted as $x_{t,\varepsilon} = \phi_{t,\varepsilon} X$, and is stochastic in time $t$ and deterministic in the parameter $\varepsilon$. Since the two parameters $t$ and $\varepsilon$ are independent, we may compute the partial derivative of either parameter, while holding the other one fixed. Moreover, since $t$ and $\varepsilon$ are independent parameters, we may take partial derivatives with respect to these parameters in either order and equate their cross derivatives. 

In this situation of two-parameter diffeomorphisms, the family of Lagrangian trajectories \eqref{stochdiff-Lag-traj} has been extended to include their deterministic deformations. This extension is expressed as,
\begin{align}\label{stochdiff-Lag-traj-var}
\diff x_{t,\varepsilon} = u_{t,\varepsilon}(x_{t,\varepsilon}) {\,\diff t} + \xi(x_{t,\varepsilon}) \circ {\diff W_t}
\,.
\end{align}
We define two time-dependent vector fields $w_t(x_t)$ and $\delta u_t(x_t) $ in terms of the following two different types of tangents of the perturbed trajectories at the identity, $\varepsilon=0$,
\begin{align}\label{w+du-vector-fields}
\begin{split}
w_t(x_t) &:= \frac{\partial }{\partial \varepsilon}(\phi_{t,\varepsilon} X) \bigg|_{\varepsilon=0}
= \frac{\partial x_{t,\varepsilon}}{\partial \varepsilon}\bigg|_{\varepsilon=0}
\,,\\
\delta u_t(x_t) &:= \frac{\partial \,u_{t,\varepsilon}(x_{t,\varepsilon=0})}{\partial \varepsilon}\bigg|_{\varepsilon=0}
\,.
\end{split}
\end{align}
The definitions of the path \eqref{stochdiff-Lag-traj-var} and its tangent vectors at the identity with respect to $\varepsilon$ in \eqref{w+du-vector-fields} lead to the following Lemma, which will be useful in proving the Euler--Poincar\'e theorem for stochastic continuum dynamics in the next subsection. 
\begin{lemma} [Velocity variations]\label{prep-lemma}
The variational vector fields $\delta u_t(x_t) $ and $w_t(x_t)$ defined in \eqref{w+du-vector-fields} satisfy the following advective transport relation:
\begin{align}\label{stochdiff-Lag-traj-lemma1}
\delta u_t(x_t) {\,\diff t} =  \diff w_t + \pounds_{\diff x_t}w_t 
=  \diff w_t + \big[\diff x_t\,,\,w_t\big] =:  \diff w_t - {\rm ad}_{\,\diff x_t} w_t 
\,.
\end{align}
\end{lemma}
\begin{remark}\rm
The advective transport relation \eqref{stochdiff-Lag-traj-lemma1} in Lemma \ref{prep-lemma} implies that the variation of the velocity vector field $\delta u_t(x_t)$ is determined by integrating the pull-back of the stochastic flow process $\phi_t$ acting on the infinitesimal deformation vector field, $w_t$, as 
\begin{align}\label{stochdiff-Lag-traj-lemma2}
\delta u_t(x_t) {\,\diff t} = \,\diff \,(\phi_t^*w_t) =  \phi_t^*(\diff w_t + \mathcal L_{\diff x_t}w_t )
\,,
\end{align}
where $x_t=\phi_t X$ and $\diff x_t$ is given in equation \eqref{stochdiff-Lag-traj}. Equation \eqref{stochdiff-Lag-traj-lemma2} is an example of the type of result which is obtained from the KIW formula in \eqref{KIWkformsimplified}, which clearly also applies for vector fields. 
\end{remark}
\begin{proof}
The proof of Lemma \ref{prep-lemma} follows from equality of cross derivatives of the smooth map $\phi_{t,\varepsilon}$ with respect to its two independent parameters, $t$ and $\varepsilon$, when the latter parameter is evaluated at the identity $\varepsilon=0$. One calculates directly that
\begin{align}\label{var-lemma-proof}
\begin{split}
\diff \left[\frac{\partial }{\partial \varepsilon}(\phi_{t,\varepsilon} X) \right]_{\varepsilon=0}
&= \diff w_t(x_t) 
+ \left[\frac{\partial w_{t,\varepsilon} }{\partial x_{t,\varepsilon}} \diff x_{t,\varepsilon}\right]_{\varepsilon=0}
= \diff w_t(x_t) 
+ \frac{\partial w_t }{\partial x_t} \cdot \diff x_{t}
\,,\\
\left[\frac{\partial }{\partial \varepsilon} \diff\,(\phi_{t,\varepsilon} X) \right]_{\varepsilon=0}
&:=
\delta u_t(x_t) dt + \left[\frac{\partial (\diff x_{t,\varepsilon})}{\partial x_{t,\varepsilon}}
\frac{\partial }{\partial \varepsilon}(\phi_{t,\varepsilon} X) 
\right]_{\varepsilon=0}
=
\delta u_t(x_t) dt + \frac{\partial(\diff x_t) }{\partial x_t}  \cdot w_t 
\,.
\end{split}
\end{align}
Taking the difference between these two equalities then yields 
equation \eqref{stochdiff-Lag-traj-lemma1} of Lemma \ref{prep-lemma}. 

\end{proof}

\begin{definition}
The operation $\diamond: V\times V^*\to \mathfrak{X}^*$ between tensor space elements $a\in V^\ast$ and
$b\in V$  produces an element of  $\mathfrak{X}({\mathcal{D}})^\ast$, a one-form density, given by
\begin{equation}\label{continuumdiamond}
\Big\langle b \diamond a, {u}\Big\rangle_\mathfrak{X}
= -\int_{\mathcal{D}} b \cdot \pounds_{u}\,a
=: \Big\langle b\,,\, -\mathcal L_{u}\,a  \Big\rangle_V
\;,
\end{equation}
where $\langle\, \cdot\,,\,\cdot\,\rangle_\mathfrak{X}$ denotes the symmetric, non-degenerate $L^2$ pairing between  vector fields and one-form densities, which are dual with respect to this pairing. Likewise, $\langle\, \cdot\,,\,\cdot\,\rangle_V$ represents the corresponding $L^2$ pairing between elements of $V$ and $V^*$. 
Also, $\mathcal L_{u}a$ stands for the Lie derivative of an element $a\in V^*$ with respect to a vector field ${u}\in \mathfrak{X}({\mathcal{D}}),$ and $b\cdot \mathcal L_{u}\,a$ denotes the contraction between elements of $V$ and elements of $V^\ast$. 
\end{definition}

For a stochastic Stratonovich path $x_t = \phi_t X$ with $\phi_t \in \operatorname{Diff}(\mathcal{D})$, let 
\begin{align}\label{stochdiff-Lag-path}
\diff x_t = u_t(x_t) {\,\diff t} + \xi(x_t) \circ {\diff W_t}
\end{align}
be its corresponding process and consider the curve $a_t$ with initial condition $a_0$
determined by the \emph{stochastic transport equation}
\begin{equation}
\diff \big(\phi_t^*a_t \big) = \phi_t^* \Big(\diff a_t + \mathcal L_{\diff{x}_t} a_t \Big) = 0,
\label{continuityequation}
\end{equation}
which is another application of the dynamical KIW formula in \eqref{KIWkformsimplified}. \smallskip

\noindent
We can now state the Stochastic Euler--Poincar\'{e} Theorem for Continua.

\begin{theorem}[Stochastic Euler--Poincar\'{e} Theorem for Continua]
\label{EPforcontinua}
Consider a stochastic Stratonovich path $x_t = \phi_t X$ with $\phi_t \in \operatorname{Diff}(\mathcal{D})$. The following two statements are equivalent:

\begin{enumerate}[(i)]
\item Hamilton's variational principle in
Eulerian coordinates
\begin{equation}\label{continuumconstrainedVP}
 \delta S :=  \delta\!\! \int_{t_1}^{t_2} l({u},a)\ dt=0
\end{equation}
holds on $\mathfrak{X}(\mathcal{D}) \times V^\ast$, using
variations of the form given in equation \eqref{stochdiff-Lag-traj-lemma1}
\begin{equation}\label{continuumvariations}
   \delta {u}{\,\diff t} = \diff w - {\rm ad}\,_{dx_t} w\,, \qquad
   \delta a = - \mathcal L_w\,a,
\end{equation}
where the vector field $w$ vanishes at the endpoints in time, $t_1$ and $t_2$.
\item The Euler--Poincar\'{e} equations for continua, namely the auxiliary advection equation \eqref{stochdiff-Lag-path} and the following equation of motion, 
\begin{equation}\label{continuumEP}
   \diff \frac{\delta l}{\delta {u}}
   = -\, \operatorname{ad}^{\ast}_{\diff x_t}\frac{\delta l}
        {\delta {u}}
   +\frac{\delta l}{\delta a}\diamond a{\,\diff t}
   =-\mathcal L_{\diff x_t} \frac{\delta l}{\delta {u}}
    +\frac{\delta l}{\delta a}\diamond a {\,\diff t}
    \,,
\end{equation}
hold, where the $\diamond$ operation given by
(\ref{continuumdiamond}) needs to be determined on a
case by case basis, since it depends on the nature of the tensor $a$.
(Recall that $\delta l/\delta {u}$ is a one-form density.)
\end{enumerate}
\end{theorem}
\begin{proof}
The following string of equalities
shows that (i) is equivalent to (ii):
\begin{align}\label{continuumEPderivation}
\begin{split}
0 &= \delta \int_{t_1}^{t_2} l({u}, a) dt
   =\int_{t_1}^{t_2}\left(\frac{\delta l}{\delta {u}}\cdot
\delta{u} +\frac{\delta l}{\delta a}\cdot \delta a\right)dt
\\
  &= \int_{t_1}^{t_2} \left[\frac{\delta l}{\delta {u}}
  \cdot \left(\diff w 
    -\operatorname{ad}_{\,\diff{x}_t}{w}\right)
    -\frac{\delta l}{\delta a}\cdot \mathcal L_{w}\, a {\,\diff t}\right]
\\
   &= \int_{t_1}^{t_2} {w}\cdot
\left[-\,\diff\frac{\delta l}{\delta{u}}
   - \operatorname{ad}^*_{\diff x_t}\frac {\delta l}{\delta{u}}
   +\frac{\delta l}{\delta a} \diamond a {\,\diff t}\right]
   \\
   &= \int_{t_1}^{t_2} {w}\cdot
\left[-\,\diff\frac{\delta l}{\delta{u}}
   - \mathcal L_{\diff x_t}\frac {\delta l}{\delta{u}}
   +\frac{\delta l}{\delta a} \diamond a {\,\diff t}\right]\,.
\end{split}
\end{align}
In the second line of the calculation in \eqref{continuumEPderivation}, we have substituted 
the constrained variations in equation \eqref{continuumvariations}.
In the third line, we have used the product rule for the stochastic differential $(\diff\,)$ and applied  homogeneous 
endpoint conditions for $w$ under integration by parts using \eqref{stochdiff-Lag-int1}. 
\end{proof}
\noindent
The KIW formula is also essential in formulating and proving the following theorem for the Lagrange-to-Euler pull-back of the Clebsch variational principle for stochastic fluids appearing in \cite{holm2015variational}.

\begin{theorem}[Lagrange--Clebsch variational principle for stochastic continuum dynamics]$\,$
\label{stochClebschforcontinua}

Consider a cylindrically stochastic Stratonovich path $x_t = \phi_t X$ with $\phi_t \in \operatorname{Diff}(\mathcal{D})$. The following two statements are equivalent:

\begin{enumerate}[(i)]
\item 
The Clebsch-constrained Hamilton's variational principle
\begin{equation}\label{ClebschconstrainedVP}
 \delta S :=  \delta\!\! \int_{t_1}^{t_2} l ({\phi_t^*}u,{\phi_t^*}a)
 +  \Big\langle {\phi_t^*}b\,,\,{\diff}\, ({\phi_t^*}a) \Big\rangle_V \ dt=0
 \,,
\end{equation}
holds on $\mathfrak{X}(\mathcal{D}) \times V^\ast$.
\item The Euler--Poincar\'{e} equations for continua hold, in the form
\begin{align}\label{ClebschEquations}
\begin{split} 
   \diff \left( {\phi_t^*}\frac{\delta l}{\delta {u}}\right) 
   &= 
   {\phi_t^*}\left(\diff \frac{\delta l}{\delta {u}}
   +\mathcal L_{\diff x_t} \frac{\delta l}{\delta {u}}\right) 
    = {\phi_t^*}\left(\frac{\delta l}{\delta a}\diamond a\right) {\,\diff t}\,,
    \\
    \diff \big(\phi_t^*a_t \big) &= \phi_t^* \Big(\diff a_t + \mathcal L_{\diff{x}_t} a_t \Big) = 0
    \,.
\end{split}
\end{align}
\end{enumerate}
\end{theorem}
\begin{proof}
Evaluating the variational derivatives at fixed time $t$ and coordinate $X$ yields the following relations:
\begin{align}\label{ClebschVariations}
\begin{split}
\delta ({\phi_t^*}b)\ &: 0 = {\diff}\, ({\phi_t^*}a) = \phi_t^* \Big(\diff a_t + \mathcal L_{\diff{x}_t} a_t \Big),
\\
\delta ({\phi_t^*}a)\ &: 0 = -\,{\diff}\, ({\phi_t^*}b) +  {\phi_t^*}\left(\frac{\delta l}{\delta a}\diamond a\right) {\,\diff t}\,,
\\
\delta ({\phi_t^*}u)\ &: 0 =  \frac{\delta l}{\delta ({\phi_t^*}u)} - ({\phi_t^*}b)\diamond ({\phi_t^*}a)
\,.
\end{split}
\end{align}
One then computes the motion equation to be
\begin{align}\label{ClebschMotionEquation}
\begin{split}
{\diff}\,\frac{\delta l}{\delta ({\phi_t^*}u)} &= {\diff}\,({\phi_t^*}b)\diamond ({\phi_t^*}a)
+ ({\phi_t^*}b)\diamond {\diff}\,({\phi_t^*}a)
\\&=  {\phi_t^*}\left(\frac{\delta l}{\delta a}\diamond a\right)\diamond ({\phi_t^*}a) {\,\diff t}
\,.
\end{split}
\end{align}
Then, assembling the results of this computation yields the equations in \eqref{ClebschEquations}.
\end{proof}

\begin{remark}\rm
Note that the stochastic equations for continuum dynamics in Theorem \ref{EPforcontinua}  and Theorem \ref{stochClebschforcontinua} are equivalent, since the second set of resulting equations is the pull-back of the first one by the Lagrange-to-Euler map. 
\end{remark}

\begin{remark}\rm
At this point, one may proceed to recover stochastic transport versions of all of the deterministic continuum dynamics models with advected quantities derived in \cite{holm1998euler}. In doing so, one would also obtain the corresponding stochastic versions of all of their Kelvin--Noether theorems. In each case, given the Lagrangian $l:\mathfrak{X}(\mathcal{D}) \times V^\ast \rightarrow \mathbb{R}$, the Kelvin--Noether quantity is given by the circulation integral
\begin{align}\label{KN-quantity}
I(\gamma_t, {u}, a) = \oint_{\gamma_t} \frac{1}{\rho}\frac{\delta l}{\delta {u}}\,,
\end{align}
around a material loop $\gamma_t$ moving with the stochastic velocity $dx_t = u{\,\diff t} + \xi(x) \circ {\diff W_t},$ in which the quantity  ${\rho}^{-1}({\delta l}/{\delta {u}})$ is the circulation one-form integrand and the circulation integral evolves according to 
\begin{align}\label{KN-circthm}
\diff I(\gamma_t, {u}, a) 
=
\oint_{\gamma_t} \Big(\diff + \mathcal L_{\diff x_t}\Big) \left( \frac{1}{\rho}\frac{\delta l}{\delta {u}}\right)
= \oint_{\gamma_t} \frac{1}{\rho}\frac{\delta l}{\delta {a}} \diamond a \,,
\end{align}
which in the deterministic case becomes the classical Kelvin's circulation theorem. As we will see in Section \ref{sec-consequences}, the proof of Kelvin's circulation theorem for stochastic fluid dynamics is yet another application of the Kunita--It\^o--Wentzell formula in \eqref{KIWkformsimplified}. The It\^o versions of the Stratonovich stochastic equations in \eqref{continuityequation} and \eqref{continuumEP} follow by standard methods. 

The differences between the deterministic and stochastic continuum dynamics equations will always be that, while the geometric structure in each case will be preserved (including, e.g., Lie--Poisson brackets and Casimirs) the stochastic versions will introduce stochastic advection by Lie Transport (SALT). In the Lie-Poisson Hamiltonian formulations of these equations, the Hamiltonian function will be stochastic in the form
\begin{align}\label{Ham-split}
\diff H = H(m,a){\,\diff t} + \langle m\,,\, \xi(x) \rangle \,\circ {\diff W_t}
\,,
\end{align} 
where, as a result of the Legendre transformation,
\begin{align}\label{m-u-def}
m = \frac{\delta l(u,a)}{\delta{u}}
\quad\hbox{and}\quad
\frac{\delta \diff H(m,a)}{\delta{m}} = u{\,\diff t} + \xi(x) \circ {\diff W_t}
\,,
\end{align} 
the equations of motion will adopt the semidirect-product Lie Poisson form,
\begin{align}\label{SDP-LPB}
\diff F( m, a) 
= 
\Big\{ F\,,\, \diff H \Big\}
= - \left\langle (m,a) , \left[ \frac{\delta F}{\delta (m,a)}
\,,\, 
\frac{\delta \diff H}{\delta (m,a)}\right] \right\rangle,
\end{align} 
where $\langle \,\cdot\,,\,\cdot\,\rangle: \mathfrak{g}^*\times \mathfrak{g}\to \mathbb{R}$ is the $L^2$ pairing 
of the semidirect-product Lie algebra  $\mathfrak{g}$ with its dual $\mathfrak{g}^*$, and 
$[ \,\cdot\,,\,\cdot\,]: \mathfrak{g}\times \mathfrak{g}\to \mathfrak{g}$ is the Lie algebra bracket. 
This semidirect-product Lie Poisson Hamiltonian form of the equations is given in a more explicit 
matrix operator form as
\begin{align}\label{SDP-LPHamform}
\diff 
\begin{bmatrix}
m \\[5pt]  a
\end{bmatrix}
= -
\begin{bmatrix}
\big(\partial_j m_i + m_j \partial_i \big) \Box \  &\  \Box \,\diamond a 
\\[5pt]
\mathcal L_\Box  \,a   \  &\   0
\end{bmatrix}
\begin{bmatrix}
\delta \diff H/ \delta m_j \\[5pt] \delta \diff H / \delta a
\end{bmatrix}
,
\end{align} 
where $\Box$ denotes where the Lie Poisson bracket operations in \eqref{SDP-LPB} are applied. 
Note that the deterministic energy Hamiltonian $H(m,a)$ is not preserved, because in general $\{H,\diff H\}\ne0$.
The It\^o versions of the Stratonovich these stochastic fluid equations follow by standard methods. 

Lie--Poisson Hamiltonian formulations of stochastic fluid dynamics extend the finite dimensional theory of stochastic Hamiltonian systems introduced for symplectic manifolds in \cite{bismut1982mecanique} and then generalised to Poisson manifolds in \cite{lazaro2007stochastic}.  Variational integrators for stochastic motion on the Lie group $SO(3)$ were developed in \cite{bou2009stochastic}. Stochastic coadjoint motion for geometric mechanics in finite dimensions also discussed in detail in \cite{arnaudon2018stochastic,cruzeiro2018momentum}.

\end{remark}

\paragraph{\bf Example: Adiabatic compressible stochastic MHD} In the case of adiabatic
compressible stochastic magnetohydrodynamics (MHD), the action in Hamilton's
principle \eqref{continuumconstrainedVP} is given by
\begin{equation}
{S} =\int \, l(\mathbf{u},D,s,\mathbf{B}){\,\diff t} = \int \ \left(\frac{ D}{2}\
|\mathbf{u}|^2 - De(D,{s})-\frac{1}{2}|\mathbf{B}|^2\right) d^3x{\,\diff t}\,.
\label{mhdact}
\end{equation}
Here, $\mathbf{u}$ is the fluid velocity vector in equation \eqref{stochdiff-Lag-path} and $\mathbf{B}$ is the flux of the magnetic field. Geometrically, the vector $\mathbf{B}$ comprises the components of an exact two-form
\begin{equation}
\mathbf{B}\cdot d\mathbf{S} = d(\mathbf{A}\cdot d\mathbf{x}) = {\rm curl}\mathbf{A}\cdot d\mathbf{S},
\label{B-def}
\end{equation}
so that $\nabla\cdot\mathbf{B}=0$. 
The fluid's internal energy per unit mass is denoted as $e(D,{s}),$ and its dependence
on the mass density $D$ and entropy per unit mass ${s}$ is provided by the ``equation of
state", which for an isotropic medium satisfies the Thermodynamic First Law, in the form 
\begin{equation}\label{1stLaw}
de=-p\,d(1/D)+Td{s}
\,,
\end{equation}
with pressure $p(D,{s})$ and
temperature $T(D,{s})$. The variations of the Lagrangian $l$ in (\ref{mhdact}) yield
Hamilton's principle for stochastic MHD as
\begin{equation}
0 =\delta {S} = \int 
D\mathbf{u}\cdot\delta\mathbf{u}-DT\delta {s} +\left(\frac{1}{2}
|\mathbf{u}|^2 - h(p,s)\right)\delta D - \mathbf{B}\cdot\delta\mathbf{B}{\,\diff t}\, d^3x\,.
\end{equation}
The quantity $h=e+p/D$ denotes the enthalpy  per unit mass, which satisfies the thermodynamic relation 
\begin{equation}\label{enthalpy-def}
dh=(1/D)dp+Td{s}
\,,
\end{equation}
as a result of the First Law \eqref{1stLaw}.
The Euler--Poincar\'e formula in Kelvin-Noether form
(\ref{continuumEP}) yields the stochastic MHD motion equation as
\begin{equation}
\left({\diff } + \mathcal L_{\diff x_t}\right)
\left({\mathbf{u}}\cdot d\mathbf{x}\right)
- (Td{s})dt
+ \Big(\frac{ 1}{ D}\mathbf{B}\times{\rm curl}\ \mathbf{B}\cdot d\mathbf{x}\Big)dt
- \left(d\Big(\frac{1}{2} |\mathbf{u}|^2 - h\Big)\right)dt = 0\,,
\label{SALT-MHD-motion-eqn}
\end{equation}
or, in three dimensional vector form,  
\begin{equation}
{\diff }\,\mathbf{u} + ({\diff \mathbf{x}_t}\cdot\nabla)\mathbf{u} + (\nabla \mathbf{u})^T\cdot {\diff \mathbf{x}_t}
+ \Big(\frac{ 1}{ D}\nabla p\Big)dt
+\Big(\frac{ 1}{ D}\mathbf{B}\times{\rm curl}\ \mathbf{B}\Big)dt = 0\,.
\end{equation}
where 
\begin{equation}\label{vec-Lag-traj}
\diff \mathbf{x}_t := \mathbf{u}(t,\mathbf{x}_t){\,\diff t} + \boldsymbol{\xi}(\mathbf{x}_t)\circ {\diff W_t}
\end{equation}
is the stochastic Lagrangian trajectory.

By definition, the advected variables $\{{s}, \mathbf{B}, D\}$ satisfy the following Lie-derivative relations which close the ideal MHD
system, by applying the KIW formula for the advective dynamics, 
\begin{align}
\begin{split}
\left({\diff }\,+ \mathcal L_{\diff x_t}\right) {s}=0, 
&\quad{\rm or}\quad
{\diff }\,s
= -\ {\diff \mathbf{x}_t}\cdot\nabla\,{s}\,,
\\[5pt]
\left({\diff }\,+ \mathcal L_{\diff x_t}\right)(\mathbf{B}\cdot d\mathbf{S}) = 0, 
&\quad{\rm or}\quad
{\diff }\,\mathbf{B}
=  {\rm curl}\,({\diff \mathbf{x}_t}\times\mathbf{B}),
\\[5pt] 
\left({\diff }\,+ \mathcal L_{\diff x_t}\right) (D \,d^3x) = 0\,,
&\quad{\rm or}\quad
{\diff }\,D
= -\ \nabla\cdot(D\,{\diff \mathbf{x}_t})\,,
\end{split}
\label{MHD-SALT-eqns}
\end{align}
and the pressure is a function $p(D,{s})= D^2\partial e/\partial D$ specified by giving
the equation of state of the fluid, $e=e(D,{s})$. If the divergence-free condition
$\nabla\cdot\mathbf{B}=0$ holds initially, then it holds for all
time; since this constraint is preserved by the stochastic advection equation
for $\mathbf{B}$. 

The Stratonovich equations \eqref{SALT-MHD-motion-eqn} - \eqref{MHD-SALT-eqns} for stochastic MHD preserve several integral quantities, provided ${\diff \mathbf{x}_t}$ and $\mathbf{B}$ both have no normal components on the  boundary. Two of these are magnetic helicity and entropy 
\begin{align}
\Lambda_{mag} = \int \mathbf{B} \cdot {\rm curl}^{-1}\mathbf{B} \, d^3x,
\qquad
{\mathcal S} = \int D \Phi(s) \, d^3x
\,.
\label{MHD-maghelicity}
\end{align}
The two preserved integral quantities $\Lambda_{mag}$ and ${\mathcal S}$ in \eqref{MHD-maghelicity} are Casimir functions for the Lie--Poisson bracket in \eqref{SDP-LPHamform}. This means that they are preserved for every Hamiltonian. An additional conservation law exists in the special cases of isentropic $(\nabla s = 0)$ and isothermal flow $(\nabla T = 0)$ stochastic MHD. The additional conserved quantity $\Lambda_X=\int \mathbf{u} \cdot \mathbf{B} \, d^3x$ is the called the ``cross helicity'' and its stochastic evolution satisfies 

\begin{align}
\diff \Lambda_X = \diff \int \mathbf{u} \cdot \mathbf{B} \, d^3x = \int T\,\mathbf{B} \cdot \nabla s \, d^3x \,,
\label{MHD-maghelicity}
\end{align}
for the stochastic Hamiltonian
\[
 \diff H =  \int \ \left(\frac{ 1}{2D} |\mathbf{m}|^2 + De(D,{s}) + \frac{1}{2}|\mathbf{B}|^2\right) d^3x{\,\diff t}
 +
 \int  \mathbf{m} \cdot \boldsymbol{\xi}(\mathbf{x}) \circ {\diff W_t},
\]
subject to the First Law \eqref{1stLaw}.

\subsection*{Distinctions of the present approach from other approaches}
The results in this section are distinct from the related results of \cite{arnaudon2014stochastic}, \cite{holm2015variational} and \cite{chen2015constrained} in many ways. The closest relation of the present work is with \cite{holm2015variational}, since Theorem \ref{EPforcontinua} does in fact recover all of the equations derived in \cite{holm2015variational} from the Clebsch constrained variational approach in the Eulerian representation. However, the variational approach in \cite{holm2015variational} is purely Eulerian, while the present approach deals directly with stochastic Lagrangian trajectories. The results of \cite{arnaudon2014stochastic} and \cite{chen2015constrained} may be regarded as similar in spirit to the present work, because of their variational basis in the Lagrangian fluid description. However, (i) the objectives of the latter two papers differ from the present work, (ii) they use different variational procedures and (iii) they use different Lagrangians. All three of these differences lead to different dynamical equations from those derived in the present approach. First, the objectives of \cite{chen2015constrained} are to derive Navier--Stokes PDE, while we are deriving SPDE which preserve the geometric properties of deterministic ideal fluid dynamics. Second, the variation $x_{t,\epsilon}$ in \cite{chen2015constrained} is given by a composition of maps $e_t^\epsilon(x_t )$ which does not solve a stochastic differential equation (SDE), while the variations here are defined as two-parameter smooth maps $x_{t,\varepsilon} = \phi_{t,\varepsilon} X$ which satisfy the SDE in equation \eqref{stochdiff-Lag-traj-var}. Moreover, the Lagrangian trajectories in \cite{chen2015constrained} have fixed amplitude which would correspond to the special case $\partial_x\xi=0$ for the velocity vector field decomposition in our equation \eqref{VF-decomp}. In addition, the drift velocity $u$ is not determined in \cite{chen2015constrained} from stationarity under arbitrary variations, $\delta u$. Third, the paper \cite{chen2015constrained} and the present work choose different Lagrangians. Namely, paper \cite{chen2015constrained} chooses stochastic Lagrangian functionals whose variations lead to stochastic momenta, while the Lagrangian functionals in the present work are the same as in the deterministic case and the variations of the Lagrangian particle trajectories in equation \eqref{stochdiff-Lag-traj-var} are stochastic. The result is that in \cite{chen2015constrained} the variational derivatives produce stochastic momenta, whereas for us the Lagrangian paths which produce advective transport are stochastic. Thus, the difference in the choice of Lagrangians also leads to different dynamics. 

Two other prominant recent approaches to stochastic fluid dynamics which differ from the present work include that of Mikulevicious and Rozovskii \cite{mikulevicius2004stochastic,mikulevicius2005global} and those of M\'emin \textit{et al.}\cite{memin2014fluid, resseguier2017geophysical1, resseguier2017geophysical2, resseguier2017geophysical3, resseguier2017stochastic}. These two separate approaches each start with Newton's Law of particle motion and introduce a stochastic Lagrangian trajectory as in \eqref{vec-Lag-traj}. However, they then take different approaches, do not invoke either Hamilton's principle, or the KIW formula. Moreover, the equations derived in these two Newtonian approaches differ both from each other and from those derived in the present approach.

\subsection*{Outlook for the rest of the paper.}
The present section has demonstrated that the variational derivation of the class of stochastic fluid dynamics equations considered here depends vitally on the KIW formula \eqref{KIWkformsimplified}. Indeed, when the Lagrangian in equation \eqref{continuumEP} is chosen to be the kinetic energy of Euler's fluid equations for incompressible flow, one recovers the 3D SPDE stochastic Euler fluid equations which have been shown in \cite{crisan2017solution} to preserve the corresponding analytical properties of their deterministic counterparts. 

The purpose of the remainder of the present paper will be to investigate some additional implications of the Kunita--It\^o--Wentzell formula for stochastic fluid dynamics and to characterise the analytical requirements under which this KIW formula is valid. 

\section{Extension of Kunita-It\^o-Wentzell formula to $k$-forms}\label{sec-KIW-for-k-forms}
We say that a $k$-form $K(x)$ is of differentiability class $C^r\left(\bigwedge^k(\mathbb R^n)\right)$ if every component $K_{i_1,\ldots,i_k}(x)$ is $r$-times differentiable. We also define the $L^p$ norm of $k$-forms by
\begin{align*}
\|K\|_{L^p} := \left(\int_{\mathbb R^n} |K(x)|^p \diff^n x\right)^{\frac1p},
\end{align*}
for $p < \infty$ and
\begin{align*}
\|K\|_{L^\infty} := \sup_{x \in \mathbb R^n} |K(x)|,
\end{align*}
if $p=\infty$, where the norm $|\cdot|$ is given by
\begin{align*}
|K(x)| := \sqrt{\delta^{i_1 j_1} \cdots \delta^{i_k j_k} K_{i_1,\ldots,i_k}(x) K_{j_1,\ldots,j_k}(x)},
\end{align*}
where $\{\delta^{ij}\}_{i,j=1,\ldots,k}$ are the components of the Euclidean cometric tensor, which is one if $i=j$ and zero if $i \neq j$. Moreover, sum over repeated indices is assumed. We say that a $k$-form $K(x)$ is of integrability class $L^p\left(\bigwedge^k(\mathbb R^n)\right)$ if $\|K\|_{L^p} < \infty$.

We now give the statement of our main theorem. Namely, we determine precise conditions under which the Kunita-It\^o-Wentzell formula in \eqref{Ito-Wentzell-one-form-Ito-ver} below holds for $k$-form-valued diffusion processes on $\mathbb R^n$.

\begin{theorem}[Kunita-It\^o-Wentzell (KIW) formula for $k$-forms: It\^o version] \label{thm:IW-one-form-Ito-ver} $\,$\\
Let $K(t,x) \in L^\infty \left([0,T]; C^2\left(\bigwedge^k(\mathbb{R}^n)\right)\right)$ be a continuous adapted semimartingale taking values in the $k$-forms
\begin{align} \label{SPDE-ito}
    K(t,x) = K(0,x) + \int^t_0 G(s,x) \,\diff s + \sum_{i=1}^M \int^t_0 H_i(s,x) \diff W_s^i, \quad t \in [0,T],
\end{align}
where $W_t^1, \ldots, W_t^M$ are i.i.d. Brownian motions, $G \in L^1\left([0,T]; C^2\left(\bigwedge^k(\mathbb R^n)\right)\right)$ and \\
$H_i \in L^2\left([0,T]; C^2\left(\bigwedge^k(\mathbb R^n)\right)\right), i=1,\ldots, M$ are $k$-form-valued continuous adapted semimartingales.
Let $\{\phi_t\}_{t \in [0,T]}$ be a continuous adapted solution of the diffusion process
\begin{align} \label{flow-eq}
    \diff \phi_t(x) = b(t,\phi_t(x)) \,\diff t + \sum_{i=1}^N \xi_i(t,\phi_t(x)) \circ \diff B_t^i, \quad \phi_0(x) = x,
\end{align}
which is assumed to be a $C^1$-diffeomorphism, where $B_t^1,\ldots,B_t^N$ are i.i.d. Brownian motions, $b(t,\cdot) \in W^{1,1}_{loc}(\mathbb R^n,\mathbb R^n),$ $\xi_i(t,\cdot) \in C^2(\mathbb R^n,\mathbb R^n), i=1,\ldots,N$ for all $t \in [0,T]$ and $\int^T_0 |b(s,\phi_s(x)) + \frac12\sum_i \xi_i \cdot \nabla \xi_i(s,\phi_s(x))| + \sum_i|\xi_i(s,\phi_s(x))|^2 \diff s < \infty$ for all $x \in \mathbb R^n$. Then, the following fomula holds.
\begin{align} \label{Ito-Wentzell-one-form-Ito-ver}
\begin{split}
    \phi_t^* K(t,x) &=  K(0,x) + \int^t_0 \phi_s^* G(s,x) \,\diff s + \sum_{i=1}^M \int^t_0 \phi_s^* H_i(s,x) \diff W_s^i \\
    &+ \int^t_0 \phi_s^* \mathcal L_b K (s,x) \,\diff s 
    + \sum_{j=1}^N \int^t_0 \phi_s^* \mathcal L_{\xi_j} K(s,x) \diff B_s^j\,, \\
&+ \sum_{i=1}^M \sum_{j=1}^N \int^t_0 \phi^*_s \mathcal L_{\xi_j} H_i(s,x) \, \diff \,[W^i, B^j]_s + \sum_{j=1}^N \frac12 \int^t_0 \phi^*_s \mathcal L_{\xi_j} \mathcal L_{\xi_j} K(s,x) \, \diff s.
\end{split}
\end{align}
\end{theorem}
\begin{remark}
We note that equation \eqref{SPDE-ito} is defined as an It\^o integral but the flow equation \eqref{flow-eq} is given by a Stratonovich equation. The Stratonovich form in \eqref{flow-eq} is taken merely to simplify the expressions. Of course, one can obtain similar expressions using the flow equation in It\^o form by simply replacing $b \rightarrow b - \frac12 \xi \xi'$.
\end{remark}

We will defer the full proof of Theorem \ref{thm:IW-one-form-Ito-ver} to Section \ref{sec-proofs}. Here, we will sketch the proof in the case of scalar fields, following \cite{krylov2011ito} and illustrate how we could extend it to $k$-forms and vector fields, which we will discuss in more details in the full proof in Section \ref{sec-proofs}.

\begin{proof}[Sketch proof of Theorem \ref{thm:IW-one-form-Ito-ver} for scalar fields.]
We will only prove the case $N=M=1$ here for simplicity. Extension to more noise terms is straightforward. Let $K: \mathbb R^n \rightarrow \mathbb R$ be a scalar function satisfying the assumptions in Theorem \ref{thm:IW-one-form-Ito-ver} and let $\rho^\epsilon$ be a sequence of mollifiers. For any $t \in [0,T]$ and $x,y \in \mathbb R^n$, consider the following process
\begin{align} \label{F-eq-scalar}
F^\epsilon(t,x,y) = \rho^\epsilon(y - \phi_t(x)) K(t,y),
\end{align}
where $\phi_t$ is the flow of the SDE \eqref{flow-eq}. By It\^o's lemma, we have
\begin{align*}
\diff \rho^\epsilon(y - \phi_t(x)) = - \frac{\partial \rho^\epsilon}{\partial y^k} \circ \diff \phi_t^k(x) = - \left(b^k \frac{\partial \rho^\epsilon}{\partial y^k} + \frac12 \xi^k \frac{\partial}{\partial x^k}\left(\xi^l \frac{\partial \rho^\epsilon}{\partial y^l}\right)\right)\diff t - \xi^k \frac{\partial \rho^\epsilon}{\partial y^k} \diff B_t
\end{align*}
and by the stochastic product rule, we get
\begin{align*}
\diff F^\epsilon(t,x,y) =& \rho^\epsilon(y - \phi_t(x)) \diff K(t,y) + K(t,y) \diff \rho^\epsilon(y - \phi_t(x)) + \diff \left[\rho^\epsilon(y - \phi_\cdot(x)), K(\cdot,y)\right]_t \\
=& \left(\rho^\epsilon(y - \phi_t(x)) G(t,y) - K(t,y) \left(b^k \frac{\partial \rho^\epsilon}{\partial y^k} + \frac12 \xi^k \frac{\partial}{\partial x^k}\left(\xi^l \frac{\partial \rho^\epsilon}{\partial y^l}\right)\right)(t,\phi_t(x)) \right) \diff t \\
&+ \rho^\epsilon(y - \phi_t(x)) H(t,y) \diff W_t - K(t,y) \xi^k \frac{\partial \rho^\epsilon}{\partial y^k}(t,\phi_t(x)) \diff B_t - H(t,y) \xi^k \frac{\partial \rho^\epsilon}{\partial y^k} \diff \left[W,B\right]_t.
\end{align*}

Next, we integrate the $F^\epsilon$ equation with respect to $y$ and remove the derivatives on $\rho^\epsilon$ by integrating by parts. This step gives us
\begin{align*}
&\int_{\mathbb R^n} F^\epsilon(t,x,y) \diff^n y - \int_{\mathbb R^n} F^\epsilon(0,x,y) \diff^n y \\
&= \int^t_0 \int_{\mathbb R^n} \rho^\epsilon(y - \phi_s(x)) \left(G(s,y) + b^k \frac{\partial K}{\partial y^k} + \frac12 \xi^k \frac{\partial \xi^k}{\partial x^k} \frac{\partial K}{\partial y^l} + \frac12 \xi^k \xi^l \frac{\partial^2 K}{\partial y^k \partial y^l}\right)\diff^n y \diff s \\
&+ \int^t_0 \int_{\mathbb R^n}\rho^\epsilon(y - \phi_s(x)) \left(H(s,y) \diff W_s + \xi^k \frac{\partial K}{\partial y^k} \diff B_t\right) \diff^n y  + \int^t_0 \int_{\mathbb R^n}\rho^\epsilon(y - \phi_s(x)) \xi^k \frac{\partial H}{\partial y^k} \diff^n y \diff \left[W,B\right]_s,
\end{align*}
where we have assumed that Fubini's theorem can be applied. Now, taking the limit $\epsilon \rightarrow 0$ on both sides and using the dominated convergence theorem for It\^o integrals, we obtain
\begin{align*}
K(t,\phi_t(x)) - K(0,x) =& \int^t_0 \left(G(s,\phi_s(x)) + \mathcal L_b K(s, \phi_s(x)) + \frac12 \mathcal L_\xi \mathcal L_\xi K (s,\phi_s(x))\right) \diff s \\
&+ \int^t_0 H(s,\phi_s(x)) \diff W_s + \int^t_0 \mathcal L_\xi K(s, \phi_s(x)) \diff B_s + \int^t_0 \mathcal L_\xi H(s, \phi_s(x)) \diff \left[W,B\right]_s,
\end{align*}
as expected, where $\mathcal L_b K = b \cdot \nabla K$ is the Lie derivative for scalar fields.

To prove \eqref{Ito-Wentzell-one-form-Ito-ver} for $k$-forms (see full proof in Section \ref{sec-proofs}), we follow through a similar argument, except we adapt \eqref{F-eq-scalar} by setting
\begin{align*}
F^\epsilon(t,x,y) := \rho^\epsilon(y - \phi_t(x)) \left< K(t,y), (\phi_t)_* \boldsymbol{u}(\phi_t(x)) \right>,
\end{align*}
where $\boldsymbol{u} = (u_1,\ldots,u_k) \in \mathfrak X(\mathbb R^n)^{k}$ are $k$ arbitrary vector fields and $\left<\cdot,\cdot\right>$ denotes the contraction of tensors. By contracting with arbitrary vector fields, one can keep track of how the basis vectors for $K(t,x)$ transform under the pullback. Using a slight modification, we can also show that \eqref{Ito-Wentzell-one-form-Ito-ver} holds for vector fields $K \in \mathfrak{X}(\mathbb R^n)$, by setting
\begin{align*}
F^\epsilon(t,x,y) := \rho^\epsilon(y - \phi_t(x)) \left< K(t,y), (\phi_t)_* {\alpha}(\phi_t(x)) \right>,
\end{align*}
where $\alpha \in \Gamma(\bigwedge^1(\mathbb R^n))$ is an arbitrary one-form.
\end{proof} 

Next, we show that the Stratonovich version of the Kunita-It\^o-Wentzell formula for $k$-forms follows as a corollary of the previous theorem.

\begin{theorem}[Kunita--It\^o--Wentzell (KIW) formula for $k$-forms: Stratonovich version] \label{thm:IW-one-form-Strat-ver}$\,$\\
Let $K(t,x) \in L^\infty \left([0,T]; C^3\left(\bigwedge^k(\mathbb{R}^n)\right)\right)$ be a $k$-form-valued continuous adapted semimartingale satisfying the Stratonovich SPDE
\begin{align} \label{K-eq-strat}
    K(t,x) = K(0,x) + \int^t_0 G(s,x) \,\diff s + \sum_{i=1}^M \int^t_0 H_i(s,x) \circ \diff W_s^i, \quad t \in [0,T],
\end{align}
where $W_t^i$ are i.i.d. Brownian motions, $G \in L^1\left([0,T]; C^3\left(\bigwedge^k(\mathbb R^n)\right)\right)$ and \\
$H_i \in L^\infty\left([0,T]; C^3\left(\bigwedge^k(\mathbb R^n)\right)\right), i=1,\ldots,M$ are $k$-form-valued continuous adapted semimartingales such that
\begin{align}
H_i(t,x) = H_i(0,x) + \int^t_0 g(s,x) \,\diff s + \sum_{j=1}^S \int^t_0 h_{ij}(s,x) \diff N^{ij}_s, \quad t \in [0,T], \quad i=1,\ldots,M
\end{align}
satisfies the assumptions in Theorem \ref{thm:IW-one-form-Ito-ver} and $N_s^{ij}$ are i.i.d. Brownian motions.
Let $\{\phi_t\}_{t \in [0,T]}$ be a continuous adapted solution of the diffusion process
\begin{align} \label{flow-eq}
    \diff \phi_t(x) = b(t,\phi_t(x)) \,\diff t + \sum_{i=1}^N \xi_i(t,\phi_t(x)) \circ \diff B_t^i, \quad \phi_0(x) = x,
\end{align}
which is assumed to be a $C^1$-diffeomorphism, where $B_t^i$ are i.i.d. Brownian motion, $b(t,\cdot) \in W^{1,1}_{loc}(\mathbb R^n,\mathbb R^n)$ for all $t \in [0,T]$, $\xi_i \in L^\infty\left([0,T]; C^3(\mathbb R^n,\mathbb R^n)\right)$ and $\int^T_0 |b(s,\phi_s(x)) + \frac12\sum_i \xi_i \cdot \nabla \xi_i(s,\phi_s(x))| + \sum_i |\xi_i(s,\phi_s(x))|^2 \diff s < \infty$ for all $x \in \mathbb R^n$. Then, the following holds.
\begin{align} \label{Ito-Wentzell-one-form-Strat-ver}
\begin{split}
    \phi_t^* K(t,x) &=  K(0,x) + \int^t_0 \phi_s^* G(s,x) \,\diff s + \sum_{i=1}^M \int^t_0 \phi_s^* H_i(s,x) \circ \diff W_s^i \\
    &+ \int^t_0 \phi_s^* \mathcal L_b K (s,x) \,\diff s 
    + \sum_{i=1}^N \int^t_0 \phi_s^* \mathcal L_{\xi_i} K(s,x) \circ \diff B_s^i\,.
\end{split}
\end{align}
\end{theorem}
\begin{proof}[Proof of Theorem \ref{thm:IW-one-form-Strat-ver}]
In It\^o form, \eqref{K-eq-strat} is given by
\begin{align*}
K(t,x) = K(0,x) + \int^t_0 G(s,x) \,\diff s + \sum_{i} \int^t_0 H_i(s,x) \diff W_s^i + \frac12 \sum_{i,j} \int^t_0 h_{ij}(s,x) \diff \,[W^i,N^{ij}]_s.
\end{align*}
Now, applying \eqref{Ito-Wentzell-one-form-Ito-ver} on $K$, we get
\begin{align}
\phi_t^* K(t,x) &=  K(0,x) + \int^t_0 \phi_s^* G(s,x) \,\diff s + \frac12 \sum_{i,j} \int^t_0 \phi_s^* h_{ij}(s,x) \diff \,[W^i,N^{ij}]_s + \sum_{i} \int^t_0 \phi_s^* H_i(s,x) \diff W_s^i \nonumber \\
    &+ \int^t_0 \phi_s^* \mathcal L_b K (s,x) \,\diff s 
    + \sum_{i} \int^t_0 \phi_s^* \mathcal L_{\xi_i} K(s,x) \diff B_s^i\,, \nonumber \\
&+ \sum_{i,j} \int^t_0 \phi^*_s \mathcal L_{\xi_j} H_i(s,x) \, \diff \,[W^i, B^j]_s + \frac12 \sum_{i} \int^t_0 \phi^*_s \mathcal L_{\xi_i} \mathcal L_{\xi_i} K(s,x) \, \diff s. \label{K-eq-Ito}
\end{align}
On the other hand, since $H_i$ and $\mathcal L_{\xi_j} K$ satisfy the assumptions in Theorem \ref{thm:IW-one-form-Ito-ver} for $i=1,\ldots,M$ and $j=1,\ldots,N$, applying \eqref{Ito-Wentzell-one-form-Ito-ver} to $H_i$ and $\mathcal L_{\xi_i} K$ respectively gives us
\begin{align*}
\phi_t^* H_i (t,x) &=  H_i(0,x) + \int^t_0 \phi_s^* g(s,x) \,\diff s + \sum_{j} \int^t_0 \phi_s^* h_{ij}(s,x) \diff N_s^{ij} \nonumber \\
    &+ \int^t_0 \phi_s^* \mathcal L_b H_i (s,x) \,\diff s 
    + \sum_{j} \int^t_0 \phi_s^* \mathcal L_{\xi_j} H_i(s,x) \diff B_s^j\,, \nonumber \\
&+ \frac12 \sum_{j,k} \int^t_0 \phi^*_s \mathcal L_{\xi_j} h_{ik}(s,x) \, \diff \,[N^{ik}, B^j]_s + \frac12 \sum_{j} \int^t_0 \phi^*_s \mathcal L_{\xi_j} \mathcal L_{\xi_j} H_i(s,x) \, \diff s \\
\phi_t^* \mathcal L_{\xi_i} K(t,x) &=  \mathcal L_{\xi_i} K(0,x) + \int^t_0 \phi_s^* \mathcal L_{\xi_i} G(s,x) \,\diff s + \frac12 \sum_{j,k} \int^t_0 \mathcal L_{\xi_i} h_{jk}(s,x) \diff \,[W^j,N^{jk}]_s \nonumber \\
&+ \sum_{j} \int^t_0 \phi_s^* \mathcal L_{\xi_i} H_j(s,x) \diff W_s^j 
    + \int^t_0 \phi_s^* \mathcal L_b \mathcal L_{\xi_i} K (s,x) \,\diff s 
    + \sum_j \int^t_0 \phi_s^* \mathcal L_{\xi_j} \mathcal L_{\xi_i} K(s,x) \diff B_s^j \,, \nonumber \\
&+ \sum_{j,k} \int^t_0 \phi^*_s \mathcal L_{\xi_j} \mathcal L_{\xi_i} H_k(s,x) \, \diff \,[W^k, B^j]_s + \frac12 \sum_j \int^t_0 \phi^*_s \mathcal L_{\xi_j} \mathcal L_{\xi_j} \mathcal L_{\xi_i} K(s,x) \, \diff s.
\end{align*}
Therefore we get
\begin{align*}
\left[\phi_t^* H_i(\cdot,x), W^i\right]_t &= \sum_j \int^t_0 \phi_s^* h_{ij}(s,x) \diff \,[W^i,N^{ij}]_s + \sum_j \int^t_0 \phi_s^* \mathcal L_{\xi_j} H_i(s,x) \diff \,[W^i,B^j]_s \\
\left[\phi_t^* \mathcal L_{\xi_i} K(\cdot,x), B^i\right]_t &= \sum_j \int^t_0 \phi_s^* \mathcal L_{\xi_i} H_j(s,x) \diff \,[W^j,B^i]_s + \int^t_0 \phi_s^* \mathcal L_{\xi_i} \mathcal L_{\xi_i} K(s,x) \diff s.
\end{align*}
Now, it is easy to check that \eqref{K-eq-Ito} is identical to \eqref{Ito-Wentzell-one-form-Strat-ver}.
\end{proof}

\section{Implications of KIW in stochastic fluid dynamics}\label{sec-consequences}
In stochastic fluid dynamics, besides the momentum one-form density, one encounters several types of advected $k$-forms such as mass density and magnetic field. Consider a fluid in a domain $D\subset\mathbb{R}^n$ and an arbitrary control volume $\Omega_0\subset D$ with spatial coordinates $X$. We assume the fluid particles which are initially at point $X$ evolve under the flow $\phi_t$, determined by the stochastic differential equation
\begin{align} \label{flow-eq-X}
\diff \phi_t(X) = b(t,\phi_t(X))\diff t + \xi(t,\phi_t(X))\circ \diff W_t\,.
\end{align}  
The control volume at time $t$ is given by $\Omega_t = \phi_t(\Omega_0)$, and the position at time $t$ of the fluid particle initially at $X$ is denoted by $x(t;X) = \phi_t(X)$.
Assuming that a $k$-form $\alpha$ satisfies a diffusion equation of the form \eqref{K-eq-strat} and its integral over the control volume $\Omega_t$ is conserved with time, i.e.,
\begin{equation}
\int_{\Omega_t} \alpha(t,x) - \int_{\Omega_0}\alpha(0,X) = 0,
\label{kformconservation}
\end{equation} 
then the previous equation can be rewritten as
\begin{align} \label{alpha-int}
\int_{\Omega_0}\left((\phi_t^*\alpha)(t,X)- \alpha(0,X)\right) = 0,
\end{align}
after changing variables. We can then apply the KIW formula for $k$-forms \eqref{KIWkformsimplified} in the integrand of \eqref{alpha-int} to obtain
\begin{align*}
\int_{\Omega_0}\left((\phi_t^*\alpha)(t,X) - \alpha(0,X)\right) = \int_0^t \int_{\Omega_0} \phi_s^*\Big(\diff \alpha + \mathcal L_b \alpha ds + \mathcal L_{\xi}\alpha\circ dW_s\Big)(s,X) = 0.  
\end{align*}
Transforming back the coordinates yields
\begin{align*}
\int_0^t \int_{\Omega_s} \left(\diff \alpha(s,x) + \mathcal L_b \alpha(s,x) ds + \mathcal L_{\xi}\alpha(s,x)\circ dW_s\right) = 0,
\end{align*}
which implies that $\alpha$ satisfies the SPDE
\begin{equation}
\diff \alpha(s,x) + \mathcal L_b\alpha(s,x) ds + \mathcal L_{ \xi}\alpha(s,x)\circ dW_s= 0,
\label{kformequation}
\end{equation}
since the control volume $\Omega_0$ was chosen arbitrarily,
\\
\paragraph{\bf Example: Conservation of fluid mass.}
As a common application of the above, we use the conservation of mass to derive a stochastic counterpart of the continuity equation in $\mathbb R^3$. Let $D(0,\boldsymbol{X})$ be the mass density in the reference configuration and let $D(t,\boldsymbol{x})$ be the mass density at time $t$, where we employ the bold font notation ``$\boldsymbol x$" and ``$\boldsymbol X$" to denote the coordinate expression of the points $x$ and $X$ respectively. Conservation of mass reads
\begin{align*}
\int_{\Omega_t} D(t,\boldsymbol {x}) d^3\boldsymbol x -\int_{\Omega_0} D(0,\boldsymbol{X})d^3\boldsymbol X = 0.
\end{align*}
By applying the argument above, we obtain \eqref{kformequation} for the particular case of $\alpha(t,x) = D(t,\boldsymbol{x})d^3\boldsymbol x$, whose Lie derivative is expressed in coordinates as $\mathcal L_{b} (D(t,\boldsymbol{x})d^3 \boldsymbol x) = \nabla\cdot(D(t,\boldsymbol{x})\boldsymbol{b}(t,\boldsymbol{x}))d^3 \boldsymbol x$. Thus, we arrive at the stochastic continuity equation
\begin{align} \label{eq-cont}
\diff D (s,\boldsymbol{x}) + \nabla\cdot (D(s,\boldsymbol{x}) \boldsymbol{b}(s,\boldsymbol{x}))\diff s + \nabla\cdot (D(s,\boldsymbol{x}) \boldsymbol \xi(s,\boldsymbol{x})) \circ \diff W_s = 0.
\end{align}

Similar equations can be derived for the advection of entropy per unit mass, which is a scalar function, and magnetic field, which is a two-form, thus recovering all the equations given in \eqref{MHD-SALT-eqns}.
\\
\paragraph{\bf Kelvin's Circulation Theorem}
Here, we will use the KIW theorem for $k$-forms to prove that the stochastic Euler-Poincar\'e equation \eqref{continuumEP} obtained in Section \ref{sec-StochEPthm} satisfies a stochastic Kelvin's circulation theorem.

\begin{theorem}[Stochastic Kelvin's circulation theorem]
Suppose that an arbitrary material loop $c_0$ is advected by the stochastic flow $\phi_t$ solving the SDE \eqref{flow-eq-X}, and denote by $c_t = \phi_t(c_0)$ the loop at time $t$. Assume that we have a stochastic Euler-Poincar\'e equation \eqref{continuumEP} for the Lagrangian $l(u,D,\alpha) = \int_{\mathbb R^n} \left(\frac12 |u|^2 - V(\alpha)\right)D \diff\,^n x$, such that 
\begin{itemize}
\item $\alpha$ is a $k$-form satisfying the advection equation \eqref{kformequation},
\item $V(\alpha)$ is a smooth function usually representing the potential energy, and
\item The mass density $\rho := D(t,x) d^{n} x$ solves the continuity equation \eqref{eq-cont}.
\end{itemize}
Then the Kelvin-Noether quantity $v(t,x) := \rho^{-1}(t,x) \, (\delta\ell/\delta u)(t,x)$ satisfies the SPDE
\begin{align} \label{v-eq}
\diff v(t,x) = -\mathcal L_{\diff x_t} v(t,x)  + \rho^{-1} F(t,x) \diff t,
\end{align}
where  $\rho^{-1}F := \rho^{-1}(\delta\ell/\delta \alpha) \diamond \alpha + \rho^{-1}(\delta\ell/\delta \rho) \diamond \rho$ is the force per unit mass, which is a one-form, and the stochastic Kelvin's circulation theorem reads
\begin{align} \label{kelvin}
\oint_{c_t} v(t,x) - \oint_{c_0} v(0,X) &= \int_0^t \oint_{c_s} \rho^{-1}(s,x) F(s,x) \diff s \,.
\end{align}
\end{theorem}
\begin{proof}
From the Euler-Poincar\'e equation \eqref{continuumEP}, we have
\begin{align*}
\diff \left(\rho v\right) &= \rho \!\circ\! \diff v + v \!\circ\! \diff \rho = -\mathcal L_{\diff x_t}\left(\rho v\right) + F(t,x) \diff t \\
&= \rho \left(-\mathcal L_{\diff x_t} v+ \rho^{-1} F(t,x) \diff t \right) - v \mathcal L_{\diff x_t} \rho \diff t,
\end{align*}
from which we deduce \eqref{v-eq} since $\rho$ satisfies the continuity equation \eqref{eq-cont}.
By changing variables, we can express the LHS of \eqref{kelvin} as
\begin{align*}
\oint_{c_t} v(t,x) - \oint_{c_0} v(0,X) &= \oint_{c_0} \left(\phi_t^*v(t,X) - v(0,X)\right)\,.
\end{align*}
By applying the KIW formula \eqref{KIWkformsimplified} under the integral sign, we obtain
\begin{align*}
\oint_{c_0} \left(\phi_t^*v(t,X) - v(0,X)\right)  &= \oint_{c_0} \int_0^t \phi_s^*\left(\diff v + \mathcal L_b v \,ds + \mathcal L_{\xi} v \circ dW_s\right)(s,X) \\
&= \int_0^t \oint_{c_s} \left( \diff v + \mathcal L_b v \,ds + \mathcal L_{\xi} v \circ dW_s\right)(s,x),
\end{align*}
where we changed back the coordinates and applied Fubini theorem.
Finally, from the momentum equation \eqref{v-eq}, we conclude
\begin{align*}
\oint_{c_t} v(t,x) - \oint_{c_0}v(0,X) &= \int_0^t \oint_{c_s} \rho^{-1} F(s,\textbf{X}) \diff s.
\end{align*}
\end{proof}

\begin{remark}
From the stochastic Kelvin's circulation theorem \eqref{kelvin}, we can deduce that if the fluid is not acted on by external forces, then the circulation $I(t) = \oint_{c_t} v$ is preserved.
\end{remark}

\section{Proof of the KIW theorem} \label{sec-proofs}
In this section, we will prove the KIW theorem (in It\^o formulation) in full detail.
We start by introducing some preparatory results that are used in the proof.
\begin{theorem}[Stochastic Fubini Theorem, Krylov \cite{krylov2011ito}] \label{stofu}
Let $T \in \mathbb{R}^+,$ $\mathcal P$ be the predictable sigma algebra on $[0,\infty) \times \Omega$, and $G_t(x),$ $H_t(x)$ be real functions defined on $ [0,T] \times \mathbb{R}^n \times \Omega,$ satisfying the following properties:
\begin{enumerate}
\item{$G_t(x), H_t(x)$ are $\mathcal{P}_T-$measurable, where $\mathcal{P}_T$ is the restriction of $\mathcal{P}$ to $[0,T] \times \Omega.$}
\item{$G_t(x)$ and $H_t(x)$ satisfy
\begin{align*} 
\int_0^T (|G_t(x)| + |H_t(x)|^2) \diff t < \infty, \quad (x,\omega) \in \mathbb{R}^n \times \Omega \backslash  A, 
\end{align*}
for $A$ a set of measure zero.}
\item{$G_t(x)$ and $H_t(x)$ also satisfy
\begin{align*}
\int_0^T \int_{\mathbb{R}^n} |G_t(x)| \diff x \diff t + \int_0^T \left( \int_{\mathbb{R}^n} |H_t(x)|^2 \diff x \right)^{1/2} \diff t < \infty, \quad a.s.
\end{align*}}
\end{enumerate} 
Then the stochastic process 
\begin{align} \label{Krylov}
\int_0^t G_s(x) \diff s + \int_0^t H_s(x) \diff B_s, \quad t \in [0,T],
\end{align}
is well-defined, $\mathcal{P}_T \otimes \mathcal{B}(\mathbb{R}^n)$-measurable, and can be modified into a continuous stochastic process by only changing its values in a set of measure zero. Moreover, the stochastic integral
\begin{align*}
\int_0^t \int_{\mathbb{R}^n} H_s(x) \diff x \diff B_s
\end{align*}
is well-defined, and the following equality holds
\begin{align*}
\int_{\mathbb{R}^n} \int_0^t G_s(x) \diff s \diff x + \int_{\mathbb{R}^n} \int_0^t H_s(x) \diff B_s \diff x
= \int_0^t \int_{\mathbb{R}^n}  G_s(x)  \diff x \diff s + \int_0^t \int_{\mathbb{R}^n}  H_s(x) \diff x \diff B_s, \quad a.s. \quad t \in [0,T].
\end{align*}
\end{theorem}
The full proof of this result is provided in \cite{krylov2011ito}.

\begin{lemma}[It\^o's product rule]
Let $X_t^1, \ldots, X_t^k$ be semimartingales. Then, we have the following:
\begin{align}
\diff \left(X_t^1 \cdots X_t^k \right) &= \sum_{j = 1}^k \left(\prod_{\alpha \neq j}^k X_t^\alpha\right) \diff X_t^{j} + \frac12 \sum_{\substack{i,j = 1 \\ i \neq j}}^k \left(\prod_{\alpha \neq i,j}^k X_t^\alpha \right) \diff \left[X^i, X^j\right]_t. \label{ito-product-rule}
\end{align}
\end{lemma}
\begin{proof}
This can be proved straight-forwardly by induction.
\end{proof}

\begin{lemma}[Lie derivative of $k$-forms] \label{lie-deriv}
Given a differentiable $k$-form $K \in C^1\left(\bigwedge^k(\mathbb R^n)\right)$, and a $C^1$-vector field $u$, we have the following
\begin{align}
&\mathcal L_u K(x) (v_1,\ldots,v_k) = u^l(x) \frac{\partial K_{i_1,\ldots,i_k}}{\partial x^l}(x) v_1^{i_1} \cdots v_k^{i_k} + \sum_{p=1}^k K_{i_1,\ldots,i_k}(x) \frac{\partial u^{i_p}}{\partial x^l}(x) v_1^{i_1} \cdots v_p^l \cdots v_k^{i_k}, \label{lie-explicit} \\
\begin{split}
&\mathcal L_u \mathcal L_u K(x) (v_1,\ldots,v_k) = 
u^l(x) \frac{\partial}{\partial x^l}\left(u^m(x) \frac{\partial K_{i_1,\ldots,i_k}}{\partial x^m}(x)\right) v_1^{i_1} \cdots v_k^{i_k} \\
&\quad + \sum_{p=1}^k \left(u^l(x)\frac{\partial}{\partial x^l}\left(K_{i_1,\ldots,i_k}(x) \frac{\partial u^{i_p}}{\partial x^m}(x)\right) + u^l(x) \frac{\partial u^{i_p}}{\partial x^m}(x) \frac{\partial K_{i_1,\ldots,i_k}}{\partial x^l}(x) \right) v_1^{i_1} \cdots v_p^m \cdots v_k^{i_k} \\
&\quad + \sum_{\substack{p,q=1 \\ p \neq q}}^k K_{i_1,\ldots,i_k}(x) \frac{\partial u^{i_p}}{\partial x^l}(x) \frac{\partial u^{i_q}}{\partial x^m} v_1^{i_1} \cdots v_p^l v_q^m \cdots v_k^{i_k}, \label{double-lie-explicit}
\end{split}
\end{align}
for arbitrary vector fields $v_1,\ldots,v_k$.
\end{lemma}
\begin{proof}
The explicit formula \eqref{lie-explicit} for Lie derivatives can be found in \cite{marsden2013introduction}, Chapter 4.4, and the double Lie derivative formula \eqref{double-lie-explicit} can be deduced directly from \eqref{lie-explicit} applied twice.
\end{proof}

Now, we are ready to prove the Kunita-It\^o-Wentzell formula for $k$-forms \eqref{Ito-Wentzell-one-form-Ito-ver} in It\^o form.

\subsection{Proof of Theorem \ref{thm:IW-one-form-Ito-ver}}
For convenience, we denote the drift term of the Stratonovich-to-It\^o corrected version of \eqref{flow-eq}  by $\hat{b}^i := b^i + \frac12 \xi^j D_j \xi^i$ and set $N=M=1$ (the more general case can be proved similarly).

{\bf Step 1:} For fixed $x \in \mathbb R^n$ and any $y \in \mathbb R^n$, we consider the following (real-valued) process
\begin{align} \label{F-eq-full}
\begin{split}
F^\epsilon_t(x,y) &:= \rho^\epsilon(y - \phi_t(x)) \left< K(t,y), (\phi_t)_* \boldsymbol{u}(\phi_t(x)) \right> \\
&= \rho^\epsilon(y- \phi_t(x)) K_{i_1,\ldots,i_k}(t,y) \prod_{\alpha=1}^k J^{i_\alpha}_{j_\alpha}(t,x) u_\alpha^{j_\alpha}(x),
\end{split}
\end{align}
where $K(t,y) = K_{i_1,\ldots,i_k}(t,y) \diff y^{i_1} \wedge \cdots \wedge \diff y^{i_k}$ in coordinate expression, $u_\alpha(x) = u^i_{\alpha}(x) \partial/\partial x^i, \, \alpha = 1, \ldots, k$ are $k$ arbitrary smooth vector fields, $\rho^\epsilon(y) := \epsilon^{-n} \rho(y/\epsilon)$ is a family of mollifiers with $\text{Supp}(\rho) \subset B_{\gamma}(0)$ for some $\gamma > 0$, and $J^i_j(t,x):= D_j\phi_t^i(x)$ is a shorthand notation for the Jacobian matrix. The philosophy behind considering this process will become clearer later, but the main idea is that when we integrate $F^\epsilon_t$ with respect to the $y$ variable and take the limit $\epsilon \rightarrow 0$, we obtain the process of $K(t,\cdot)$ pulled-back by the flow $\phi_t$.

Let $\tau_1$ be the first exit time of the flow $\phi_t(x)$ leaving the ball $B_{R_1}(0)$ for some $|x|<R_1<\infty,$ and let $\tau_2$ be the first exit time of $D \phi_t(x)$ leaving the ball $B_{R_2}(0)$ with respect to the supremum norm $\|\cdot\|_\infty,$ for some $1<R_2<\infty$. Setting $\tau = \tau_1 \wedge \tau_2$, we have $|\phi_t(x)| < R_1$ and $|J^i_j(t,x)| < R_2$ for all $t < \tau$.
Once we prove that equation \eqref{Ito-Wentzell-one-form-Ito-ver} holds for all $t \in [0,\tau \wedge T]$, then we can take $R_1,R_2 \rightarrow \infty$ to show that it holds for any $t \in [0,T]$.

By It\^o's product rule, $F_t^\epsilon$ satisfies the following equation
\begin{align*}
&\diff F_t^\epsilon(x,y) \\
& = K_{i_1,\ldots,i_k}(t,y) \left(\prod_{\alpha=1}^k J^{i_\alpha}_{j_\alpha}(t,x) u_\alpha^{j_\alpha}(x)\right) \diff \rho^\epsilon(y-\phi_t(x)) + \rho^\epsilon(y-\phi_t(x))  \left(\prod_{\alpha=1}^k J^{i_\alpha}_{j_\alpha}(t,x) u_\alpha^{j_\alpha}(x)\right) \diff  K_{i_1,\ldots,i_k}(t,y) \\
&\,\,\,+ \rho^\epsilon(y-\phi_t(x)) K_{i_1,\ldots,i_k}(t,y) \diff \left(\prod_{\alpha=1}^k J^{i_\alpha}_{j_\alpha}(t,x) u_\alpha^{j_\alpha}(x)\right) + K_{i_1,\ldots,i_k}(t,y) \diff \left[\rho^\epsilon(y-\phi_\cdot(x)), \prod_{\alpha=1}^k J^{i_\alpha}_{j_\alpha}(\cdot,x) u_\alpha^{j_\alpha}(x) \right]_t\\
&\,\,\,+  \rho^\epsilon(y-\phi_t(x)) \diff \left[K_{i_1,\ldots,i_k}(\cdot,y), \prod_{\alpha=1}^k J^{i_\alpha}_{j_\alpha}(\cdot,x) u_\alpha^{j_\alpha}(x) \right]_t + \prod_{\alpha=1}^k J^{i_\alpha}_{j_\alpha}(t,x) u_\alpha^{j_\alpha}(x)\diff \left[\rho^\epsilon(y-\phi_\cdot(x)), K_{i_1,\ldots,i_k}(\cdot,y)\right]_t.
\end{align*}
Applying It\^o's product rule \eqref{ito-product-rule}, we get
\begin{align*}
\diff \left(\prod_{\alpha=1}^k J^{i_\alpha}_{j_\alpha}(t,x) u_\alpha^{j_\alpha}(x)\right) =& \sum_{p=1}^k \left(\prod_{\alpha \neq p}^k J^{i_\alpha}_{j_\alpha}(t,x) u_{\alpha}^{j_\alpha}(x) \right) u_p^{j_p}(x) \diff J^{i_p}_{j_p}(t,x) \\
&+ \frac12 \sum_{\substack{p,q = 1 \\ p \neq q}}^k \left(\prod_{\alpha \neq p,q}^k J^{i_\alpha}_{j_\alpha}(t,x) u_\alpha^{j_\alpha}(x) \right) u_p^{j_p}(x) u_q^{j_q}(x) \diff \left[J^{i_p}_{j_p}(\cdot,x),J^{i_q}_{j_q}(\cdot,x)\right]_t,
\end{align*}
and by It\^o's lemma, we obtain
\begin{align*}
\diff \rho^\epsilon(y-\phi_t(x)) = &\left(-\hat{b}^l(t,\phi_t(x)) D_l \rho^\epsilon(y-\phi_t(x)) + \frac12 \xi^l(t,\phi_t(x))\xi^m(t,\phi_t(x)) D^2_{lm}\rho^\epsilon(y-\phi_t(x)) \right) \diff t \\
&\quad + \xi^i(t,\phi_t(x)) D_i \rho^\epsilon(y-\phi_t(x)) \diff B_t,
\end{align*}
where $D \rho^\epsilon(y-\phi_t(x))$ denotes the derivative with respect to the $y$ variable. 
We differentiate \eqref{flow-eq} with respect to $x$ to derive (recall that $\phi_t$ is a $C^1$-diffeomorphism)
\begin{align*}
&\diff J_j^i(t,x) = D_l \hat{b}^i(t,\phi_t(x)) J_j^l(t,x) \diff t +  D_l \xi^i(t,\phi_t(x)) J_j^l(t,x) \diff B_t.
\end{align*}
This naturally imposes the condition $\int^T_0 \left(\|D_x \hat{b}(t,\phi_t(x))\| + \|D_x \xi(t,\phi_t(x))\|^2 \right)\diff t < \infty$.
By direct calculation, one can show that $F^\epsilon_t(x,y)$ can be expressed as
\begin{align}
\begin{split}\label{F-eq}
F^\epsilon_t(x,y) - F^\epsilon_0(x,y) = & \int^t_0 \hat{G}^{1,\epsilon}_s(x,y) \diff s + \int^t_0 \hat{G}^{2,\epsilon}_s(x,y) \diff \,[W,B]_s \\
&+ \int^t_0 \hat{H}^{1,\epsilon}_s(x,y) \diff W_s + \int^t_0 \hat{H}^{2,\epsilon}_s(x,y) \diff B_s,
\end{split}
\end{align}
for all $t \in [0,\tau\wedge T]$, where
\begin{align*}
&\hat{G}^{1,\epsilon}_s(x,y) := \left[\rho^\epsilon(y-\phi_s(x)) G_{i_1,\ldots,i_k}(s,y) + \left(-\hat{b}^l(s,\phi_s(x)) D_l \rho^\epsilon(y-\phi_s(x))  \right.\right.\\
& \left. \left. + \frac12 \xi^l(s,\phi_s(x))\xi^m(s,\phi_s(x)) D^2_{lm}\rho^\epsilon(y-\phi_s(x)) \right) K_{i_1,\ldots,i_k}(s,y)\right] \prod_{\alpha=1}^k J^{i_\alpha}_{j_\alpha}(s,x) u_\alpha^{j_\alpha}(x) \\
&+ \sum_{p=1}^k K_{i_1,\ldots,i_k}(s,y) \left(\rho^\epsilon(y-\phi_s(x)) D_l \hat{b}^{i_p}(s,\phi_s(x)) - \xi^m(s,\phi_s(x)) D_m \rho^\epsilon(y-\phi_s(x)) D_l \xi^{i_p}(s,\phi_s(x))\right)\\
&\times \left(\prod_{\alpha \neq p}^{k} J^{i_\alpha}_{j_\alpha}(s,x) u_\alpha^{j_\alpha}(x)\right) J^l_{j_p}(s,x) u_p^{j_p}(x) \\
&+ \frac12 \sum_{\substack{p,q = 1 \\ p \neq q}}^k \rho^\epsilon(y-\phi_s(s)) K_{i_1,\ldots,i_k}(s,y) D_l \xi^{i_p}(s,\phi_s(x)) D_m \xi^{i_q}(s,\phi_s(x)) \left(\prod_{\alpha \neq p,q}^{k} J^{i_\alpha}_{j_\alpha}(s,x) u_\alpha^{j_\alpha}(x)\right) \\
&\times J^l_{j_p}(s,x) J^m_{j_q}(s,x) u_p^{j_p}(x) u_q^{j_q}(x),
\end{align*}
\begin{align*}
\hat{G}^{2,\epsilon}_s(x,y) &:= - \xi^l(s,\phi_s(x)) D_l\rho^\epsilon(y-\phi_s(x)) H_{i_1,\ldots,i_k}(s,y) \prod_{\alpha = 1}^k J^{i_\alpha}_{j_\alpha}(s,x) u_\alpha^{j_\alpha}(x) \\
&\hspace{15pt}+\sum_{p=1}^k \rho^\epsilon(y-\phi_s(x)) H_{i_1,\ldots,i_k}(s,y) D_l \xi^{i_p}(s,\phi_s(x)) \left(\prod_{\alpha \neq p}^k J^{i_\alpha}_{j_\alpha}(s,x) u_\alpha^{j_\alpha}(x) \right) J^l_{j_p}(s,x) u_p^{j_p}(x),
\end{align*}
\begin{align*}
&\hat{H}^{1,\epsilon}_s(x,y) := \rho^\epsilon(y-\phi_s(x)) H_{i_1,\ldots,i_k}(s,y) \prod_{\alpha = 1}^k J^{i_\alpha}_{j_\alpha}(s,x) u_\alpha^{j_\alpha}(x),  \\
&\hat{H}^{2,\epsilon}_s(x,y) := -\xi^l(s,\phi_s(x)) D_l \rho^\epsilon(y-\phi_s(x)) K_{i_1,\ldots,i_k}(s,y) \prod_{\alpha = 1}^k J^{i_\alpha}_{j_\alpha}(s,x) u_\alpha^{j_\alpha}(x) \\
& \hspace{40pt}+ \sum_{p=1}^k \rho^\epsilon(y-\phi_s(x)) K_{i_1,\ldots,i_k}(s,y) D_l \xi^{i_p}(s,\phi_s(x)) \left(\prod_{\alpha \neq p}^k J^{i_\alpha}_{j_\alpha}(s,x) u_\alpha^{j_\alpha}(x) \right) J^l_{j_p}(s,x) u_p^{j_p}(x).
\end{align*}

{\bf Step 2:}  We integrate \eqref{F-eq} with respect to the variable $y$ on both sides and switch the order of the integrals using the stochastic Fubini theorem (Theorem \ref{stofu}). To check that the conditions in the stochastic Fubini theorem are satisfied, first note that $\hat{G}^{i,\epsilon}, \hat{H}^{i,\epsilon}, i=1,2$ are predictable, owing to the measurability conditions imposed in the assumptions. We also have
\begin{align*}
\int^{\tau \wedge T}_0 |\hat{G}^{1,\epsilon}(x,y)|\diff t  \lesssim &\,\|\rho^\epsilon\|_{L^\infty_y} \|G(\cdot,y)\|_{L^1_t} + \|K(\cdot, y)\|_{L^\infty_t}\left[ \|D \rho^\epsilon\|_{L^\infty_y} \|\hat{b}(\cdot,\phi_\cdot(x))\|_{L^1_t}
\right. \\
& + \|D^2 \rho^\epsilon\|_{L^\infty_y} \|\xi(\cdot,\phi_\cdot(x))\|^2_{L^2_t} + \|\rho^\epsilon\|_{L^\infty_y} \|D_x\hat{b}(\cdot,\phi_\cdot(x))\|_{L^1_t} \\
&\left. +\|D \rho^\epsilon\|_{L^\infty_y} \|\xi(\cdot,\phi_\cdot(x))\|_{L^2_t} \|D_x\xi(\cdot,\phi_\cdot(x))\|_{L^2_t} + \|\rho^\epsilon\|_{L^\infty_y} \|D_x\xi(\cdot,\phi_\cdot(x))\|_{L^2_t}\right], \\
\int^{\tau \wedge T}_0 |\hat{G}^{2,\epsilon}(x,y)|\diff t \lesssim& \, \|H(\cdot,y)\|_{L^2_t}\left(\|\rho^\epsilon\|_{L^\infty_y}\|D_x\xi(\cdot,\phi_\cdot(x))\|_{L^2_t} + \|D\rho^\epsilon\|_{L^\infty_y} \|\xi(\cdot,\phi_\cdot(x))\|_{L^2_t}\right), \\
\int^{\tau \wedge T}_0 |\hat{H}^{1,\epsilon}(x,y)|^2\diff t \lesssim& \, \|\rho^\epsilon\|_{L^\infty_y} \|H(\cdot,y)\|_{L^2_t}^2, \\
\int^{\tau \wedge T}_0 |\hat{H}^{2,\epsilon}(x,y)|^2\diff t \lesssim& \, \|K(\cdot,y)\|_{L^\infty_t}\left(\|\rho^\epsilon\|_{L^\infty_y}\|D_x\xi(\cdot,\phi_\cdot(x))\|_{L^2_t}^2 + \|D\rho^\epsilon\|_{L^\infty_y} \|\xi(\cdot,\phi_\cdot(x))\|_{L^2_t}^2\right),
\end{align*}
where $\|\cdot\|_{L^p_t}$ denotes the $L^p$ norm with respect to time for $t \in [0,\tau \wedge T]$, $\|\cdot\|_{L^p_y}$ denotes the $L^p$ norm with respect to space, and we used that $J^i_j(t,x) < R_2$ for all $t \in [0,\tau \wedge T]$. So for every $y \in \mathbb R^n$, the second condition is satisfied. Next, taking $D := B_{R_1 + \epsilon \gamma}(0)$, we check that
\begin{align*}
\int^{\tau \wedge T}_0 \left(\int_{\mathbb R^n}|\hat{G}^{1,\epsilon}(x,y)|\diff y\right)\diff t  \lesssim &\ \sup_{y \in D}\|G(\cdot,y)\|_{L^1_t} \|\rho^\epsilon\|_{L^\infty_y} +  \sup_{y \in D}\|K(\cdot, y)\|_{L^\infty_t}\left[ \|D \rho^\epsilon\|_{L^\infty_y} \|\hat{b}(\cdot,\phi_\cdot(x))\|_{L^1_t}
\right. \\
& + \|D^2 \rho^\epsilon\|_{L^\infty_y} \|\xi(\cdot,\phi_\cdot(x))\|^2_{L^2_t} + \|\rho^\epsilon\|_{L^\infty_y} \|D_x\hat{b}(\cdot,\phi_\cdot(x))\|_{L^1_t} \\
&\left. +\|D \rho^\epsilon\|_{L^\infty_y} \|\xi(\cdot,\phi_\cdot(x))\|_{L^2_t} \|D_x\xi(\cdot,\phi_\cdot(x))\|_{L^2_t} + \|\rho^\epsilon\|_{L^\infty_y} \|D_x\xi(\cdot,\phi_\cdot(x))\|_{L^2_t}\right], \\
\int^{\tau \wedge T}_0 \left(\int_{\mathbb R^n}|\hat{G}_t^{2,\epsilon}(x,y)| \diff y \right)\diff t \lesssim &\, \sup_{y \in D}\|H(\cdot,y)\|_{L^2_t}\left(\|\rho^\epsilon\|_{L^\infty_y}\|D_x\xi(\cdot,\phi_\cdot(x))\|_{L^2_t} + \|D\rho^\epsilon\|_{L^\infty_y} \|\xi(\cdot,\phi_\cdot(x))\|_{L^2_t}\right), \\
\int^{\tau \wedge T}_0 \left(\int_{\mathbb R^n}|\hat{H}_t^{1,\epsilon}(x,y)|^2 \diff y\right)^{\frac12}\diff t \lesssim &\, \sup_{y \in D}\|H(\cdot,y)\|_{L^2_t} \|\rho^\epsilon\|_{L^\infty_y},\\
\int^{\tau \wedge T}_0 \left(\int_{\mathbb R^n}|\hat{H}_t^{2,\epsilon}(x,y)|^2 \diff y\right)^{\frac12}\diff t \lesssim &\, \sup_{y \in D}\|K(\cdot,y)\|_{L^\infty_t}\left(\|\rho^\epsilon\|_{L^\infty_y}\|D_x\xi(\cdot,\phi_\cdot(x))\|_{L^2_t} + \|D\rho^\epsilon\|_{L^\infty_y} \|\xi(\cdot,\phi_\cdot(x))\|_{L^2_t}\right),
\end{align*}
so the third condition is also satisfied. Hence, applying the stochastic Fubini theorem and integrating by parts in $y$, we obtain
\begin{align}
\int_{\mathbb R^n}F^\epsilon_t(x,y) \diff y - \int_{\mathbb R^n} F^\epsilon_0(x,y) \diff y = & \int^t_0 \int_{\mathbb R^n} \widetilde{G}^{1,\epsilon}_s(x,y) \diff y \diff s + \int^t_0 \int_{\mathbb R^n}\widetilde{G}^{2,\epsilon}_s(x,y) \diff y \diff \,[W,B]_s \nonumber \\
&+ \int^t_0 \int_{\mathbb R^n} \hat{H}^{1,\epsilon}_s(x,y) \diff y \diff W_s + \int^t_0 \int_{\mathbb R^n}\widetilde{H}^{2,\epsilon}_s(x,y) \diff y \diff B_s, \label{F-eq-int-2}
\end{align}
where
\begin{align*}
\widetilde{G}^{1,\epsilon}_s(x,y) &:= \rho^\epsilon(y-\phi_s(x)) G_{i_1,\ldots,i_k}(s,y) \prod_{\alpha = 1}^k J^{i_\alpha}_{j_\alpha}(s,x) u_\alpha^{j_\alpha}(x) \\
&+ \rho^\epsilon(y-\phi_s(x)) \left[b^l(s,\phi_s(x)) D_lK_{i_1,\ldots,i_k}(s,y) \prod_{\alpha = 1}^k J^{i_\alpha}_{j_\alpha}(s,x) u_\alpha^{j_\alpha}(x) \right. \\
&\left. \quad + \sum_{p=1}^k K_{i_1,\ldots,i_k}(s,y) D_l b^{i_p}(s,\phi_s(x)) \left(\prod_{\alpha \neq p}^k J^{i_\alpha}_{j_\alpha}(s,x) u_\alpha^{j_\alpha}(x) \right) J^l_{j_p}(s,x) u_p^{j_p}(x)\right] \\
&+ \rho^\epsilon(y-\phi_s(x)) \left[\frac12 \xi^l(s,\phi_s(x)) \xi^m(s,\phi_s(x)) D_{lm}^2 K_{i_1,\ldots,i_k}(s,y) \prod_{\alpha = 1}^k J^{i_\alpha}_{j_\alpha}(s,x) u_\alpha^{j_\alpha}(x) \right.\\
&\quad + \frac12 \xi^m(s,\phi_s(x))D_m \xi^l(s,\phi_s(x)) D_l K_{i_1,\ldots,i_k}(s,y) \prod_{\alpha = 1}^k J^{i_\alpha}_{j_\alpha}(s,x) u_\alpha^{j_\alpha}(x) \\
&\quad + \sum_{p=1}^k \left[\xi^m(s,\phi_s(x)) D_l \xi^{i_p}(s,\phi_s(x)) D_m K_{i_1,\ldots,i_k}(s,y)\right. \\
&\quad \left.+ \frac12 D_l\left[\xi^m(s,\phi_s(x)) D_m \xi^{i_p}(s,\phi_s(x))\right] K_{i_1,\ldots,i_k}(s,y)\right]\left(\prod_{\alpha \neq p}^k J^{i_\alpha}_{j_\alpha}(s,x) u_\alpha^{j_\alpha}(x) \right) J^l_{j_p}(s,x) u_p^{j_p}(x) \\
&\quad + \frac12 \sum_{\substack{p,q=1 \\ p \neq q}}^k K_{i_1,\ldots,i_k}(s,y) D_l \xi^{i_p}(s,\phi_s(x)) D_m \xi^{i_q}(s,\phi_s(x))\left(\prod_{\alpha \neq p,q}^k J^{i_\alpha}_{j_\alpha}(s,x) u_\alpha^{j_\alpha}(x)\right) \\
& \left. \quad \times J^l_{j_p}(s,x) J^m_{j_q}(s,x) u_p^{j_p}(x) u_q^{j_q}(x)\right],
\end{align*}
\begin{align*}
\widetilde{G}^{2,\epsilon}_s(x,y) := &\rho^\epsilon(y-\phi_s(x)) \left[\xi^l(s,\phi_s(x)) D_l H_{i_1,\ldots,i_k}(s,y) \prod_{\alpha = 1}^k J^{i_\alpha}_{j_\alpha}(s,x) u_\alpha^{j_\alpha}(x) \right. \\
&\left. \hspace{50pt} \quad + \sum_{p=1}^k H_{i_1,\ldots,i_k}(s,y) D_l \xi^{i_p}(s,\phi_s(x)) \left(\prod_{\alpha \neq p}^k J^{i_\alpha}_{j_\alpha}(s,x) u_\alpha^{j_\alpha}(x) \right) J^l_{j_p}(s,x) u_p^{j_p}(x)\right],
\end{align*}
\begin{align*}
\widetilde{H}^{2,\epsilon}_s(x,y) := &\rho^\epsilon(y-\phi_s(x)) \left[\xi^l(s,\phi_s(x)) D_l K_{i_1,\ldots,i_k}(s,y) \prod_{\alpha = 1}^k J^{i_\alpha}_{j_\alpha}(s,x) u_\alpha^{j_\alpha}(x) \right. \\
&\hspace{50pt}\left. \quad + \sum_{p=1}^k K_{i_1,\ldots,i_k}(s,y) D_l \xi^{i_p}(s,\phi_s(x)) \left(\prod_{\alpha \neq p}^k J^{i_\alpha}_{j_\alpha}(s,x) u_\alpha^{j_\alpha}(x) \right) J^l_{j_p}(s,x) u_p^{j_p}(x)\right].
\end{align*}

{\bf Step 3:} Finally, we investigate the convergence of each term in the limit $\epsilon \rightarrow 0$. First, since $K$ is continuous in $y$, we obtain the following limit on the LHS of \eqref{F-eq-int-2}:
\begin{align*}
\int_{\mathbb R^n}\left(F^\epsilon_t(x,y) - F^\epsilon_0(x,y)\right) \diff y \rightarrow \left< \phi^*_t K(t,x),\boldsymbol u(x) \right> - \left<K(0,x),\boldsymbol u(x) \right>,
\end{align*}
as $\epsilon \rightarrow 0$, where $\boldsymbol u(x) = (u_1(x),\ldots,u_k(x)),$ and $\left<K(x),\boldsymbol u(x) \right> := K(x)(u_1(x),\ldots,u_k(x))$ denotes the contraction of tensors. For the terms on the RHS, we apply the dominated convergence theorem to obtain the limit. Using H\"older's inequality and noting that $\|\rho^\epsilon\|_{L^1} = 1$, we derive
\begin{align*}
&\left|\int_{\mathbb R^n} \widetilde{G}^{1,\epsilon}_s(x,y) \diff y\right| \leq \lambda^1_1(s,x) \|G(s,\cdot)\|_{L^\infty_{R_1+\gamma}} + \lambda^1_2(s,x) \|DK(s,\cdot)\|_{L^\infty_{R_1+\gamma}} + \lambda^1_3(s,x) \|D^2K(s,\cdot)\|_{L^\infty_{R_1+\gamma}}, \\
&\left|\int_{\mathbb R^n} \widetilde{G}^{2,\epsilon}_s(x,y) \diff y\right| \leq \lambda^2_1(s,x) \|H(s,\cdot)\|_{L^\infty_{R_1+\gamma}} + \lambda^2_2(s,x) \|DH(s,\cdot)\|_{L^\infty_{R_1+\gamma}},
\end{align*}
for all $\epsilon < 1$, where $\lambda_i^j(s,x)$ are locally integrable in time for $s \in [0,\tau \wedge T]$. Hence, by the dominated convergence theorem, one can show that the bounded variation parts converge as follows
\begin{align*}
\bullet \quad & \int^t_0 \int_{\mathbb R^n} \widetilde{G}^{1,\epsilon}_s(x,y) \diff y \diff s  \\
&\rightarrow \int^t_0 \int_{\mathbb R^n} \left<\phi_s^* G(s,x), \boldsymbol u(x)\right> \diff s + \int^t_0 \left<\phi_s^* \mathcal L_b K(s,x), \boldsymbol u(x)\right> \diff s + \frac12 \int^t_0 \left<\phi_s^* \mathcal L_\xi \mathcal L_\xi K(s,x), \boldsymbol u(x)\right> \diff s, \\
\bullet \quad & \int^t_0 \int_{\mathbb R^n}\widetilde{G}^{2,\epsilon}_s(x,y) \diff y \diff \,[W,B]_s \rightarrow \int^t_0 \left<\phi^*_s \mathcal L_\xi H(s,x), \boldsymbol u(x)\right>\diff \,[W,B]_s,
\end{align*}
where we have taken into account the explicit formulae for Lie derivatives in Lemma \ref{lie-deriv}.
Similarly, we can show that
\begin{align*}
&\left|\int_{\mathbb R^n} \hat{H}^{1,\epsilon}_s(x,y) \diff y\right| \leq \mu^1_1(s,x) \|H(s,\cdot)\|_{L^\infty_{R_1+\gamma}}, \\
&\left|\int_{\mathbb R^n} \widetilde{H}^{2,\epsilon}_s(x,y) \diff y\right| \leq \mu^2_1(s,x) \|K(s,\cdot)\|_{L^\infty_{R_1+\gamma}} + \mu^2_2(s,x) \|DK(s,\cdot)\|_{L^\infty_{R_1+\gamma}},
\end{align*}
where $\mu_i^j(s,x)$ are locally square integrable in time for $s \in [0,\tau \wedge T]$, so by the dominated convergence theorem for It\^o integrals, the martingale terms converge to
\begin{align*}
\bullet \quad & \int^t_0 \int_{\mathbb R^n} \hat{H}^{1,\epsilon}_s(x,y) \diff y \diff W_s \rightarrow \int^t_0 \left<\phi^*_s H(s,x), \boldsymbol u(x) \right> \diff W_s, \\
\bullet \quad & \int^t_0 \int_{\mathbb R^n} \widetilde{H}^{2,\epsilon}_s(x,y) \diff y \diff B_s \rightarrow \int^t_0 \left<\phi_s^* \mathcal L_\xi K(s,x),\boldsymbol u(x) \right>\diff B_s,
\end{align*}
in probability. Since $\boldsymbol u$ was chosen arbitrarily, this proves \eqref{Ito-Wentzell-one-form-Ito-ver} for $t \in [0,\tau \wedge T]$.

\section{Conclusions and outlook for further research}\label{sec-conclusions}

In this paper we have:

\begin{itemize}
\item Proved the Kunita--It\^o--Wentzell (KIW) formula for evaluation of stochastic $k$-forms along stochastic flows. This formula generalises the classic It\^o--Wentzell formula (see \cite{kunita1981some}, \cite{kunita1984stochastic}, \cite{kunita1997stochastic}), as well as the Kunita's It\^o lemma for $k$-forms on $\mathbb R^n$ (shown in \cite{kunita1997stochastic}).

\item Employed the KIW formula in deriving an Euler--Poincar\'e variational principle and a Clebsch constrained Hamilton's principle which each introduce stochastic advection by Lie transport (SALT) into the semidrect-product  continuum equations derived in \cite{holm1998euler} while preserving their Kelvin--Noether theorem and Lie--Poisson Hamiltonian structure.

\item Applied the KIW formula to provide a rigorous derivation of stochastic advection by Lie transport (SALT) equations, continuity equations in fluid dynamics, and Kelvin's circulation Theorem. SALT has been found to be a valuable tool in the modelling of geophysical fluid dynamics, where it enables uncertainty quantification \cite{CoCrHoPaSh2018salt1,CoCrHoPaSh2018salt2} and is expected to lead to uncertainty reduction via data assimilation. It also has been shown to play a similar important role in shape analysis \cite{ArHoSo2018shape1,ArHoSo2018shape2}. All of these results have been developed within the context of \cite{holm2015variational,CoGoHo2017,crisan2017solution}, where the geometric approach for adding SALT to deterministic fluid equations was first introduced, understood and applied.  
\end{itemize}

Some near-term future research directions may include: 
\begin{enumerate}[$\lozenge$]
\item {\bf Tensor fields.} We already know that the KIW formula is valid for vector fields, as well as $k$-forms. Extending the KIW formula to stochastic time-dependent $(r,s)$-tensor fields would provide a basis for deriving the stochastic counterparts of the deterministic transport formulas appearing in  \cite{holm1998euler}, e.g., for nonlinear elasticity.
\item {\bf Stochastic transport on manifolds.} One would expect that the KIW formula for $k$-forms (and more generally, for $(r,s)$ tensor fields) could naturally be extended to manifolds. An extensive literature about stochastic flows on manifolds exists, see. e.g., \cite{ElworthyLeJanLi2007,ElworthyLeJanLi2010}. 
The obstacle in this direction for us is that in our proof, first, one would have to make sense of \eqref{F-eq-full}, where we evaluate the $k$-form and vector fields at different points in space, which may be justified for instance by introducing a connection and taking the parallel transport to the same point. Secondly, our proof is not local since we consider mollifiers and integrate by parts, which may cause difficulty in the manifold case where we can only work locally on charts, unless we have a coordinate-free proof. However, we have good reasons to conjecture that our KIW formula does hold on manifolds, since our final expression \eqref{KIWkformsimplified} is coordinate free and it would also recover Kunita's It\^o-lemma \cite{kunita1981some,kunita1984stochastic} for $k$-forms on manifolds in the deterministic case.

\item  {\bf  A new methodology for uncertainty quantification and reduction.}
The stochastic fluid velocity decomposition results of \cite{holm2015variational} and \cite{CoGoHo2017} show that the principles of transformation theory and multi-time homogenisation comprise the foundations for a physically meaningful, data-driven and mathematically-based  approach for decomposing the fluid transport velocity into its drift and stochastic parts. This approach can be applied immediately to the class of continuum flows whose deterministic motion is based on fundamental variational principles. 

Two related papers \cite{CoCrHoPaSh2018salt2,CoCrHoPaSh2018salt1} have recently employed this approach to develop a new methodology to implement the velocity decomposition of \cite{holm2015variational} and \cite{CoGoHo2017}  for uncertainty quantification in computational simulations of fluid dynamics. The new methodology was tested numerically in these papers and found to be suitable for coarse graining in two separate types of problems based on discretisations using either finite elements, or finite differences. Preliminary results of work in progress show that combining stochastic uncertainty quantification with data assimilation can be very effective in reduction of uncertainty.

We expect that the stochastic modelling approach developed using the KIW formula in the present paper will be tenable whenever a body of hydrodynamic transport data shows the characteristic signal of high power at low frequencies. This characteristic signal is often seen in flows in Nature, such as atmospheric and oceanic geophysical flows. In such flows, the opportunity arises to decompose the corresponding Lagrangian trajectories into fast and slow, or resolvable and unresolvable, components and apply the stochastic modelling approach described here as a basis for quantifying \textit{a priori} uncertainty and then using data assimilation methods (e.g., particle filtering) for reducing uncertainty. 

\end{enumerate}

\subsection*{Acknowledgements.} We are enormously grateful to our friends and colleagues whose advice and patient encouragement have contributed greatly to our understanding of these issues. Among those to whom we are especially grateful for fruitful discussions of the present material, we thank A. Arnaudon, X. Chen, C. J. Cotter, D. Crisan, A. B. Cruzeiro, T. Drivas, F. Flandoli, F. Gay-Balmaz, S. Hochgerner, J.-M. Leahy, X.-M. Li, J. P. Ortega, W. Pan, and T. S. Ratiu.

During this work, ST was supported by the Schr\"odinger scholarship scheme at Imperial College London. ABdL and EL were supported by [grant number EP/L016613/1] and are grateful for warm hospitality at the Imperial College London EPSRC Centre for Doctoral Training in Mathematics of Planet Earth. DH was partially supported by EPSRC Standard  [Grant number  EP/N023781/1], entitled, ``Variational principles for stochastic parameterisations in geophysical fluid dynamics''.

\bibliography{biblio}
\bibliographystyle{alpha}

\end{document}